\documentclass[AIF,Unicode,manuscript]{cedram}
\usepackage[utf8]{inputenc}
\usepackage{amsmath}
\usepackage{amssymb}
\usepackage{amsthm}
\usepackage{dsfont}
\usepackage{amsfonts}
\usepackage{mathtools}
\usepackage{relsize,exscale}
\usepackage{indentfirst}
\usepackage{graphicx}
\usepackage{enumitem}
\usepackage{nicematrix}
\hypersetup{
colorlinks=true,
linkcolor=blue
}

\title[Thin surface groups in non-uniform lattices]{Zariski-dense surface groups in non-uniform lattices of split real Lie groups}
\author{\firstname{Jacques} \lastname{Audibert}}
\address{
Sorbonne Université and Université Paris Cité, CNRS, IMJ-PRG, F-75005 Paris, France.}
\email{audibert.j@outlook.fr}

\theoremstyle{plain}

\newtheorem{proposition}{Proposition}[section]
\newtheorem{theorem}[proposition]{Theorem}

\newtheorem{lemma}[proposition]{Lemma}
\newtheorem{thmx}{Theorem}
\newtheorem{propx}{Proposition}[section]
\newtheorem{corox}{Corollary}[section]

\theoremstyle{definition}

\newtheorem{definition}[proposition]{Definition}

\newtheorem{remark}[proposition]{Remark}

\newcommand{\SL}{\textnormal{SL}}
\DeclareMathOperator{\SU}{SU}
\DeclareMathOperator{\GL}{GL}
\newcommand{\G}{\textnormal{G}}
\DeclareMathOperator{\A}{A}

\DeclareMathOperator{\N}{N}
\DeclareMathOperator{\D}{D}
\DeclareMathOperator{\M}{M}
\DeclareMathOperator{\B}{B}
\DeclareMathOperator{\Nred}{N_{red}}
\newcommand{\SO}{\textnormal{SO}}
\DeclareMathOperator{\Gal}{Gal}
\DeclareMathOperator{\Aut}{Aut}
\DeclareMathOperator{\Inn}{Inn}
\DeclareMathOperator{\PSL}{PSL}
\DeclareMathOperator{\PGL}{PGL}
\DeclareMathOperator{\Br}{Br}

\DeclareMathOperator{\Det}{Det}
\DeclareMathOperator{\I}{I}

\DeclareMathOperator{\Int}{Int}
\DeclareMathOperator{\Hom}{Hom}
\DeclareMathOperator{\PSp}{PSp}

\DeclareMathOperator{\J}{J}

\DeclareMathOperator{\Diag}{Diag}

\DeclareMathOperator{\PP}{P}

\DeclareMathOperator{\X}{X}
\DeclareMathOperator{\T}{T}

\DeclareMathOperator{\HH}{H}

\DeclareMathOperator{\W}{W}
\newcommand{\Sp}{\textnormal{Sp}}

\DeclareMathOperator{\Q}{Q}
\DeclareMathOperator{\disc}{disc}
\DeclareMathOperator{\SSS}{S}

\newcommand{\R}{\mathds{R}}
\newcommand{\QQ}{\mathds{Q}}
\DeclareMathOperator{\Tr}{Tr}
\DeclareMathOperator{\MCG}{MCG}

\newcommand{\doublefaktor}[3]{%
    {\textstyle #1}
    \mkern-4mu\scalebox{1.5}{$\diagdown$}\mkern-5mu^{\textstyle #2}%
    \mkern-4mu\scalebox{1.5}{$\diagup$}\mkern-5mu{\textstyle #3} }

\makeatletter
\def\@tocline#1#2#3#4#5#6#7{\relax
  \ifnum #1>\c@tocdepth 
  \else
    \par \addpenalty\@secpenalty\addvspace{#2}%
    \begingroup \hyphenpenalty\@M
    \@ifempty{#4}{%
      \@tempdima\csname r@tocindent\number#1\endcsname\relax
    }{%
      \@tempdima#4\relax
    }%
    \parindent\z@ \leftskip#3\relax \advance\leftskip\@tempdima\relax
    \rightskip\@pnumwidth plus4em \parfillskip-\@pnumwidth
    #5\leavevmode\hskip-\@tempdima
      \ifcase #1
       \or\or \hskip 1em \or \hskip 2em \else \hskip 3em \fi%
      #6\nobreak\relax
    \dotfill\hbox to\@pnumwidth{\@tocpagenum{#7}}\par
    \nobreak
    \endgroup
  \fi}
\makeatother



\begin{document}

\begin{abstract}
    For $\SL(n,\mathds{R})$ ($n\geq3$), $\SO(n+1,n)$ ($n\geq2$), $\Sp(2n,\mathds{R})$ ($n\geq2$) and for the adjoint real split form of the exceptional group $\G_2$, we exhibit non-uniform lattices in which we construct \emph{thin Hitchin} representations by arithmetic methods. These representations give infinitely many orbits under the action of the mapping class group. In particular, we show that when $p\neq2$ is prime every non-uniform lattice of $\SL(p,\mathds{R})$ contains thin Hitchin representations.
\end{abstract}

\begin{altabstract}
    Pour $\SL(n,\mathds{R})$ ($n\geq3$), $\SO(n+1,n)$ ($n\geq2$), $\Sp(2n,\mathds{R})$ ($n\geq2$) et pour la forme réelle adjointe déployée du groupe exceptionnel $\G_2$, nous exhibons des réseaux non-uniformes dans lesquels nous construisons des représentations \emph{Hitchin fines} par des méthodes arithmétiques. Ces représentations donnent une infinité d'orbites sous l'action du groupe modulaire. En particulier, nous montrons que quand $p\neq 2$ est premier tout réseau non-uniforme de $\SL(p,\mathds{R})$ contient des représentations Hitchin fines.
\end{altabstract}  

\maketitle

\tableofcontents

\section*{Introduction}

Let $\G$ be a semi-simple algebraic Lie group. A \emph{lattice} in $\G$ is a discrete subgroup $\Gamma$ such that $\G/\Gamma$ has finite volume with respect to the Haar measure. We say that $\Gamma$ is \emph{uniform} if $\G/\Gamma$ is compact, \emph{non-uniform} otherwise. A subgroup $\Pi$ of a lattice $\Gamma$ in $\G$ is said to be \emph{thin} if it is of infinite index in $\Gamma$ and Zariski-dense in $\G$. Thin groups are an active field of research (see Sarnak \cite{Sarnak_NotesThinMatrixGroups} and \cite{Kontorovich_Whatisathingroup}). In particular, there is a good deal of interest in constructing thin surface subgroups in lattices \cite{Kassel_Groupesdesurface}. The main purpose of this paper is to find thin surface groups in non-uniform lattices of the classical split groups and the adjoint split form of $\G_2$. Out method is arithmetic and does not use Kahn-Markovic's technique \cite{Kahn_Immersingsurfacesinhyerbolicthreemanifold}.

Denote by $S_g$ a closed connected orientable surface of genus $g$ at least 2, $\pi_1(S_g)$ its fundamental group and $\MCG(S_g)=\Aut(\pi_1(S_g))/\Int(\pi_1(S_g))$ its mapping class group. The surface groups we construct are images of representations that lie in the \emph{Hitchin component} of the corresponding split group $\G$ (see Hitchin \cite{Hitchin_LiegroupsTeichmullerspace}). This is a special connected component of the character variety $\mathfrak{X}(\pi_1(S_g),\G)=\Hom(\pi_1(S_g),\G)/\G$, which for $\G=\PSL(n,\mathds{R})$ can be defined as follows.

For any $n\geq2$, denote by $\tau_{n}:\SL(2,\mathds{C})\rightarrow\SL(n,\mathds{C})$ the representation where
\begin{equation*}
    \begin{pmatrix}
    a&b\\c&d
    \end{pmatrix}
    \in \SL(2,\mathds{C})
\end{equation*}
acts on the space of homogeneous polynomials in two variables $X$ and $Y$ of degree $n-1$ as
\begin{equation*}
    \begin{pmatrix}
    a&b\\c&d
    \end{pmatrix}
    X^{n-i-1}Y^i=(aX+cY)^{n-i-1}(bX+dY)^i
\end{equation*}
for every $0\leq i\leq n-1$. We call $\tau_n$ \emph{the irreducible representation} of $\SL(2,\mathds{C})$. We denote also by $\tau_{n}:\PSL(2,\mathds{C})\rightarrow\PSL(n,\mathds{C})$ the induced representation on $\PSL(2,\mathds{C})$. The Hitchin component is the connected component of $\mathfrak{X}(\pi_1(S_g),\PSL(n,\mathds{R}))$ that contains the equivalence class of a representation of the form $\tau_n\circ j$ with $j:\pi_1(S_g)\to\PSL(2,\mathds{R})$ faithful and discrete. We call \emph{Hitchin representation} a representation whose equivalence class is in the Hitchin component. An important feature of Hitchin representations is the following result.

\begin{theorem}[Labourie \cite{Labourie_Anosovflowssurfaceandcurves}, Fock and Goncharov \cite{Fock_ModuliSpacesofLocalsystems}]
Hitchin representations are discrete and faithful.
\end{theorem}

Faithfulness of Hitchin representations plays a crucial role in our construction. We say that a Hitchin representation is \emph{thin} if its image is a thin subgroup of a lattice of $\G$. The question we address is: which lattices of $\G$ contain the image of a thin Hitchin representation?

The lattices we will consider are \emph{$\mathds{Q}$-arithmetic subgroups} (we introduce this notation to simplify the exposition). A \emph{$\mathds{R}/\mathds{Q}$-form} of $\G$ is a $\mathds{Q}$-algebraic group that is isomorphic to $\G$ as an $\mathds{R}$-algebraic group. We say that two subgroups $\Gamma_1$ and $\Gamma_2$ of $\G(\mathds{R})$ are \emph{commensurable} if $\Gamma_1\cap\Gamma_2$ has finite index in both $\Gamma_1$ and $\Gamma_2$. If $\HH$ is a $\mathds{Q}$-algebraic subgroup of $\GL_n$, we define its \emph{$\mathds{Z}$-points} as $\HH(\mathds{Q})\cap\GL_n(\mathds{Z})$. This latter depends on the embedding of $\HH$ in $\GL_n$, but two different embeddings give rise to commensurable subgroups.

\begin{definition}
A $\mathds{Q}$-arithmetic subgroup is a subgroup of $\G$ that is conjugate to a subgroup commensurable with the integer points of a $\mathds{Q}$-form of $\G$.
\end{definition}
Such groups are lattices of $\G$ as shown by Harish-Chandra. They are arithmetic groups, but not all arithmetic groups are $\mathds{Q}$-arithmetic. However, if $\G$ is simple and not compact, every non-uniform arithmetic group is $\mathds{Q}$-arithmetic (see Morris's book \cite{Morris_IntroductionArithmeticGroups} corollary 5.3.2).

In this paper we find thin Hitchin representations with image in some $\mathds{Q}$-arithmetic subgroups of the Lie groups $\SL(n,\mathds{R})$, $\SO(k+1,k)$, $\Sp(2n,\mathds{R})$ and $\G_2$. We start by fixing a cocompact $\mathds{Q}$-arithmetic subgroup of $\SL(2,\mathds{R})$. Let $a,b\in\mathds{N}^*$ and define \begin{equation*}
    \Gamma_{a,b}=\left\{\begin{pmatrix}
    x_0+\sqrt{a}x_1&\sqrt{b}x_2+\sqrt{ab}x_3\\
    \sqrt{b}x_2-\sqrt{ab}x_3&x_0-\sqrt{a}x_1
    \end{pmatrix}\ |\ x_i\in\mathds{Z},\ \Det=1\right\}.
\end{equation*}
This is a $\mathds{Q}$-arithmetic subgroup of $\SL(2,\mathds{R})$ (and, up to conjugation and commensurability, all $\mathds{Q}$-arithmetic subgroups of $\SL(2,\mathds{R})$ are of this form). The lattice $\Gamma_{a,b}$ is cocompact if and only if $ax^2+by^2=1$ has no solution $(x,y)\in\mathds{Q}^2$. If it is cocompact, then $\mathds{H}^2/\Gamma_{a,b}$ admits a finite cover which is a closed surface of genus at least 2 and we construct representations of the fundamental group of this surface.

We begin by classifying the $\mathds{Q}$-arithmetic subgroups $\Lambda$ of $\G$ that contain $\tau_n(\Gamma_{a,b})$ up to finite index, where $\G$ is one of the Lie groups above. This is the core of the proof and uses nonabelian Galois cohomology.
The group $\tau_n(\Gamma_{a,b})$ is not Zariski-dense in $\G$ since it lies in $\tau_n(\SL(2,\mathds{R}))$. We then ``bend" the representation enough so that it sill lies in $\Lambda$ but is Zariski-dense.
The bending technique was introduced by Johnson and Millson \cite{Johnson_DeformationSpacesCompactHyperbolicManifolds} and has already been used to construct thin subgroups, for example in Long and Thistlethwaite \cite{Long_ZariskidensesurfaceSL2k+1Z} or Ballas and Long \cite{Ballas_ConstructingThinSubgroupsSLn+1RBending}. To bend, we need to find a simple closed curve on the surface which has a big enough centralizer in $\Lambda$. For $\Lambda=\SL(n,\mathds{Z})$, we fail to find such a curve.
Using nonabelian Galois cohomology, we are able to compute the centralizer of this curve in $\Lambda$ explicitly. Since the deformed representations lie in the Hitchin component, we know that they are still faithful. Their image will be thus a thin surface subgroup of $\Lambda$ (the lattices considered here cannot be commensurable to a surface group as they have virtual cohomological dimension at least $3$ \cite{Aramayona_GeometricDimensionLattices}).

We first state our results for $n$ odd. For $d\in\mathds{N}$ not a square and $\sigma\in\Gal(\mathds{Q}(\sqrt{d})/\mathds{Q})$ non-trivial define
$$\SU(\I_n,\sigma;\mathds{Z}[\sqrt{d}])=\{\M\in\SL(n,\mathds{Z}[\sqrt{d}])\ |\ \sigma(\M)^\top\M=\I_n\}$$ where $\sigma(\M)$ consists of applying $\sigma$ to all entries of $\M$ and $\M^\top$ denotes the transposition of $\M$. This is a lattice of $\SL(n,\mathds{R})$. Here is a description of the $\mathds{Q}$-arithmetic subgroups of $\SL(n,\mathds{R})$ that contain $\tau_n(\Gamma_{a,b})$.

\begin{propx}
\label{proposition2}
    Let $\Gamma$ be a $\mathds{Q}$-arithmetic subgroup of $\SL(2,\mathds{R})$ and $n\geq3$ be odd.
    Then $\tau_n(\Gamma)$ lies in a subgroup of $\SL(n,\mathds{R})$ commensurable with a conjugate of $\SL(n,\mathds{Z})$ and in a subgroup commensurable with a conjugate of $\SU(\I_n,\sigma;\mathds{Z}[\sqrt{d}])$ for every $d\in\mathds{N}$ not a square with $\sigma\in\Gal(\mathds{Q}(\sqrt{d})/\mathds{Q})$ non-trivial.
    
    Furthermore these are the only $\mathds{Q}$-arithmetic subgroups of $\SL(n,\mathds{R})$ that contain $\tau_n(\Gamma)$ up to commensurability.
\end{propx}

\begin{remark}
As a consequence, even though $\tau_n(\Gamma_{a,b})$ has coefficients in $\mathds{Z}[\sqrt{a},\sqrt{b}]$ it can be conjugated to a subgroup of $\SL(n,\mathds{Z})$.
\end{remark}

The proposition implies that the corresponding locally symmetric spaces
\begin{equation*}
    \doublefaktor{\SO(n)}{\SL(n,\mathds{R})}{\Lambda}
\end{equation*}
where $\Lambda$ is (a torsion free finite index subgroup of) one of the arithmetic subgroups listed above, contains a totally geodesic surface of irreducible type, i.e. such that the copy of the hyperbolic plane in the universal cover of the locally symmetric space comes from an irreducible embedding of $\SL(2,\mathds{R})$ in $\SL(n,\mathds{R})$. Except when $\Lambda=\SL(n,\mathbb{Z})$, we manage to bend the representation $\tau_n$ to make it Zariski-dense.

\begin{thmx}
\label{theorem1}
    Let $n\geq3$ be odd. Let $d\in\mathds{N}$ be not a square and $\sigma\in\Gal(\mathds{Q}(\sqrt{d})/\mathds{Q})$ non-trivial. There exists $g\geq2$ such that the $\mathds{Q}$-arithmetic subgroup $\SU(\I_{n},\sigma;\mathds{Z}[\sqrt{d}])$ of $\SL(n,\mathds{R})$ contains infinitely many $\MCG(S_g)$-orbits of thin Hitchin representations of $\pi_1(S_g)$.
\end{thmx}

Equivalently, the $\mathds{Q}$-arithmetic subgroup $\SU(\I_{n},\sigma;\mathds{Z}[\sqrt{d}])$ contains infinitely many $\SL(n,\mathds{R})$-conjugacy classes of Zariski-dense surface subgroups of genus $g$ which are the image of Hitchin representations. 
Theorem \ref{theorem1} should be compared to the result of Borel which states that there are only finitely many arithmetic surfaces of a given genus up to isometry (see Theorem 11.3.1 in \cite{Maclachlan_ArithmeticHyperbolic3Manifolds}). To guarantee that the representations we construct are not in the same $\MCG(S_g)$-orbit, we show that their sets of traces have different congruence properties.

Note that the group $\SL(n,\mathds{Z})$ for $n$ odd is not part of our list but has been shown to contain thin Hitchin representations by Long and Thistlethwaite \cite{Long_ZariskidensesurfaceSL2k+1Z}. This result was also announced by Burger, Labourie and Wienhard (see Theorem 24 in \cite{Wienhard_InvitationHigherTeichmuller} and \cite{Burger_OnIntegerPointHitchin}).
Together with Proposition \ref{propositionprime} below, this implies the following.

\begin{corox}
\label{corollary1}
Let $p\neq2$ be a prime. All non-uniform lattices of $\SL(p,\mathds{R})$ contain a thin Hitchin representation.
\end{corox}

For $p=3$, Long, Reid and Thistlethwaite proved in \cite{Long_ZariskidensesurfaceSL3Z} and \cite{Long_ConstructingThin} that all non-uniform lattices of $\SL(3,\mathds{R})$ contain thin Hitchin representations.

We also construct Zariski-dense surface groups in $\mathds{Q}$-arithmetic subgroups of other split Lie groups. For $p,q\in\mathds{N}$ define \begin{equation*}
\I_{p,q}=\begin{pmatrix}\I_p&\\
&-\I_q
\end{pmatrix}.\end{equation*}

\begin{thmx}
\label{theorem3}
Let $k\geq2$. For each of the following $\mathds{Q}$-arithmetic subgroups $\Lambda$ of $\SO(\I_{k+1,k},\mathds{R})$, there exists $g\geq 2$ such that $\Lambda$ contains infinitely many $\MCG(S_g)$-orbits of thin Hitchin representations of $\pi_1(S_g)$:
\begin{itemize}
    \setlength\itemsep{0cm}
    \item if $k\equiv0,3[4]:$ $\SO(\I_{k+1,k},\mathds{Z})$ 
    \item if $k\equiv1,2[4]:$ $\SO(\Q,\mathds{Z})$ for $\Q\in\SL(2k+1,\mathds{Q})$ a symmetric matrix of signature $(k+1,k)$ which is not equivalent to $\I_{k+1,k}$ over $\mathds{Q}$.
\end{itemize}
In particular, for $k\equiv1,2[4]$, every non-uniform lattice of $\SO(\I_{k+1,k},\mathds{R})$ not commensurable with a conjugate of $\SO(\I_{k+1,k},\mathds{Z})$, contains infinitely many $\MCG(S_g)$-orbits of thin Hitchin representations of $\pi_1(S_g)$ for some $g\geq2$.
\end{thmx}

This is shown by establishing an analogue of Proposition \ref{proposition2} in the case of $\SO(\I_{k+1,k},\mathds{R})$.

\begin{thmx}
\label{theorem4}
There exists $g\geq2$ such that the group $\emph{\textbf{G}}_2(\mathds{Z})$, which is a $\mathds{Q}$-arithmetic subgroup of $\emph{\textbf{G}}_2(\mathds{R})$, contains infinitely many $\MCG(S_g)$-orbits of thin Hitchin representations of $\pi_1(S_g)$.
\end{thmx}

Here $\textbf{G}_2(\mathds{R})$ is the adjoint connected split real Lie group of type $\G_2$ (see Definition \ref{definitionG2}). It appears that $\textbf{G}_2(\mathds{Z})$ is the only non-uniform lattice of $\textbf{G}_2(\mathds{R})$ up to commensurability and conjugation (see Remark \ref{remarkG2lattice}). 

Proposition \ref{proposition2} shows that for $n$ odd, the list of $\mathds{Q}$-arithmetic subgroups of $\SL(n,\mathds{R})$ that contain the group $\tau_n(\Gamma)$ up to commensurability is the same for all $\Gamma$. For $n$ even, the list depends on the group $\Gamma$. To be more precise we need to define more concepts.

The $\mathds{Q}$-arithmetic subgroups of $\SL(2,\mathds{R})$ can be described in terms of \emph{quaternion algebras}. We say that a unital associative algebra $A$ over a field $k$ is a quaternion algebra if there exist $i$ and $j$ in $A$ such that $(1,i,j,ij)$ is a basis of $A$ with $i^2=a\in k^{\times}$, $j^2=b\in k^{\times}$ and $ij=-ji$. In this case, we denote $A$ by $(a,b)_k$. Note that $(a,b)_k\simeq(b,a)_k\simeq(a,-ab)_k$ and if $x\in k^{\times}$ then $(ax^2,b)_k\simeq(a,b)_k$. We say that $A$ is a \emph{division algebra} if every non-zero element of $A$ is invertible. Quaternion algebras are endowed with an involution called the \emph{conjugation} $$A\rightarrow A,\ x=x_0+x_1i+x_2j+x_3ij\mapsto\overline{x}=x_0-x_1i-x_2j-x_3ij,$$ which allows us to define a norm $\Nred:A\to k,\ x\mapsto x\overline{x}$. Suppose that $k$ is a number field and denote by $\mathcal{O}_k$ its ring of integers. An \emph{order} of $A$ is a finitely generated $\mathcal{O}_{k}$-submodule containing a basis of $A$ over $k$, containing $1$ and which is a subring of $A$. For example if $a,b\in\mathcal{O}_{k}$ then $\mathcal{O}=\mathcal{O}_{k}[1,i,j,ij]$ is an order of $A$.

We say that $A$ \emph{splits} over a field extension $K$ if $A\otimes_{k}K\simeq\M_2(K)$. For instance, a quaternion algebra over $\mathds{Q}$ splits over $\mathds{R}$ if and only if $a$ or $b$ is positive. The $\mathds{R}/\mathds{Q}$-forms of $\SL_2$ are algebraic groups defined by $$\HH(K)=\{x\in A\otimes_{\mathds{Q}}K|\Nred(x)=1\}$$ for any field extension $K$ of $\mathds{Q}$, with $A$ a quaternion algebra over $\mathds{Q}$ that splits over $\mathds{R}$.  Its integer points are commensurable with $\mathcal{O}^1=\{x\in\mathcal{O}|\Nred(x)=1\}$ for $\mathcal{O}$ an order of $A$. We can embed $\mathcal{O}^1$ in $\SL(2,\mathds{R})$ using
$$\mathcal{O}^1\hookrightarrow A^1\hookrightarrow (A\otimes_{\mathds{Q}}\mathds{R})^1\simeq\SL(2,\mathds{R})$$
and the image will be commensurable with $\Gamma_{a,b}$, defined above. Thus $\mathds{Q}$-arithmetic subgroups of $\SL(2,\mathds{R})$ are subgroups commensurable to conjugates of $\mathcal{O}^1$, for $\mathcal{O}$ an order of a quaternion algebra $A$ that splits over $\mathds{R}$. Furthermore $\mathcal{O}^1$ is cocompact if and only if $A$ is a quaternion division algebra, which is equivalent to $ax^2+by^2=1$ having no solutions $(x,y)\in\mathds{Q}^2$.

We can now state our result for $n$ even. If $A$ splits over $\mathds{R}$, $\M_{\frac{n}{2}}(\mathcal{O})$ embeds in $\M_n(\mathds{R})$. This allows us to define $\SL(\frac{n}{2},\mathcal{O})$ as the matrices of $\M_{\frac{n}{2}}(\mathcal{O})$ which have determinant $1$ in $\M_n(\mathds{R})$. For $\sigma\in\Gal(\mathds{Q}(\sqrt{d})/\mathds{Q})$ define $$ \SU(\I_{\frac{n}{2}},\overline{\phantom{s}}\otimes\sigma;\mathcal{O}\otimes\mathds{Z}[\sqrt{d}])=\{\M\in\SL(\text{\tiny{\( \frac{n}{2}\)}},\mathcal{O}\otimes_{\mathds{Z}}\mathds{Z}[\sqrt{d}])|\partial(\M)^\top\M=\I_{\frac{n}{2}}\} $$ where $\partial:A\otimes_{\mathds{Q}}\mathds{Q}(\sqrt{d})\to A\otimes_\mathds{Q}\mathds{Q}(\sqrt{d}),\ x\otimes t\mapsto\overline{x}\otimes\sigma(t)$ and $\partial(\M)$ consists of applying $\partial$ to all entries of $\M$.
Below is a description of $\mathds{Q}$-arithmetic subgroups of $\SL(n,\mathds{R})$ that contain $\tau_n(\Gamma_{a,b})$.

\begin{propx}
\label{proposition2.1}
    Let $\Gamma$ be a $\mathds{Q}$-arithmetic subgroup of $\SL(2,\mathds{R})$ and $n\geq4$ be even. Let $\mathcal{O}$ be an order of a quaternion algebra $A$ over $\mathds{Q}$ such that $\Gamma$ is commensurable with $\mathcal{O}^1$.
    
    Then $\tau_n(\Gamma)$ lies in a subgroup of $\SL(n,\mathds{R})$ commensurable with a conjugate of
    \begin{itemize}
        \item $\SL(\frac{n}{2},\mathcal{O})$,
        \item $\SU(\I_n,\sigma;\mathds{Z}[\sqrt{d}])$ for any $d\in\mathds{N}$ not a square such that $A$ splits over $\mathds{Q}(\sqrt{d})$ with $\sigma\in\Gal(\mathds{Q}(\sqrt{d})/\mathds{Q})$ non-trivial,
        \item $\SU(\I_{\frac{n}{2}},\overline{\phantom{s}}\otimes\sigma;\mathcal{O}\otimes\mathds{Z}[\sqrt{d}])$ for any $d\in\mathds{N}$ not a square such that $A$ does not split over $\mathds{Q}(\sqrt{d})$ with $\sigma\in\Gal(\mathds{Q}(\sqrt{d})/\mathds{Q})$ non-trivial.
    \end{itemize}
    
    Furthermore these are the only $\mathds{Q}$-arithmetic subgroups of $\SL(n,\mathds{R})$ that contain $\tau_n(\Gamma)$ up to commensurability.
\end{propx}

For most arithmetic groups above, we manage to bend the representation $\tau_n$ to make it thin.

\begin{thmx}
\label{theorem2}
Let $n\geq4$ be even. For the following $\mathds{Q}$-arithmetic subgroups $\Lambda$ of $\SL(n,\mathds{R})$, there exists $g\geq2$ such that $\Lambda$ contains infinitely many $\MCG(S_g)$-orbits of thin Hitchin representations of $\pi_1(S_g)$:
\begin{itemize}
        \item $\SU(\I_n,\sigma;\mathds{Z}[\sqrt{d}])$ for any $d\in\mathds{N}$ not a square with $\sigma\in\Gal(\mathds{Q}(\sqrt{d})/\mathds{Q})$ non-trivial
        \item $\SU(\I_{\frac{n}{2}},\overline{\phantom{s}}\otimes\sigma;\mathcal{O}\otimes\mathds{Z}[\sqrt{d}])$ for any $d\in\mathds{N}$ not a square and $\mathcal{O}$ an order of a quaternion algebra over $\mathds{Q}$ that does not split over $\mathds{Q}(\sqrt{d})$ with $\sigma\in\Gal(\mathds{Q}(\sqrt{d})/\mathds{Q})$ non-trivial.
    \end{itemize}
\end{thmx}

Note that the subgroup $\SL(4,\mathds{Z})$ is not part of our list but has been shown to contain thin Hitchin representations by Long and Thistlethwaite \cite{Long_ZariskidensesurfaceSL4Z}.

\begin{remark}
    Using the technique presented in \cite{Long_ZariskidensesurfaceSL2k+1Z} and \cite{Zshornack_IntegralZariskidensesurfacegroups}, one can show that $\SL(\frac{n}{2},\mathcal{O})$ also contains thin Hitchin representations, for $\mathcal{O}$ an order of any non-split quaternion algebra over $\mathds{Q}$. However, the construction is slightly different from the one presented here.
\end{remark}

We also construct thin Hitchin representations in $\mathds{Q}$-arithmetic subgroups of $\Sp(n,\mathds{R})$.

\begin{thmx}
\label{theorem5}
Let $n\geq4$ be even. Let $\mathcal{O}$ be an order of a division quaternion algebra over $\mathds{Q}$ with conjugation $\overline{\phantom{s}}$. There exists $g\geq2$ such that the $\mathds{Q}$-arithmetic subgroup $\SU(\I_{\frac{n}{2}},\overline{\phantom{s}};\mathcal{O})$ of $\Sp(n,\mathds{R})$ contains infinitely many $\MCG(S_g)$-orbits of thin Hitchin representations of $\pi_1(S_g)$.
\end{thmx}

Together with Proposition \ref{propositionsymplectic} below, the theorem implies the following.

\begin{corox}
\label{corollary2}
Let $n\geq 4$ be even. For every non-uniform lattice $\Lambda$ of $\Sp(n,\mathds{R})$ not commensurable with a conjugate of $\Sp(n,\mathds{Z})$, there exists $g\geq2$ such that $\Lambda$ contains infinitely many $\MCG(S_g)$-orbits of thin Hitchin representations of $\pi_1(S_g)$.
\end{corox}

Again, the group $\Sp(4,\mathds{Z})$ is not part of our list but has been shown to contain thin Hitchin representations in \cite{Long_ZariskidensesurfaceSL4Z}. Together with \cite{Long_ZariskidensesurfaceSL4Z}, the corollary shows that every non-uniform lattice of $\Sp(4,\mathds{R})$ contains a thin Hitchin representation.\\

Constructing surface subgroups in lattices of Lie groups (not necessarily Zariski-dense) has been done for uniform lattices of many semisimple Lie groups. In \cite{Hamenstadt_Incompressiblesurfaceslocallysymmetricspaces}, Hamenstädt proved that any cocompact lattice of a rank one simple Lie group of non-compact type not isomorphic to $\SO(2n,1)$ contains surface subgroups, generalizing the work of Kahn and Markovic \cite{Kahn_Immersingsurfacesinhyerbolicthreemanifold}. Kahn, Labourie and Mozes \cite{Kahn_Surfacegroupsinuniformlattices} proved the same result for many other semi-simple Lie groups, notably the complex ones. It is expected that their construction provides Zariski-dense surface groups. As it is explained in \S1.2 of \cite{Kahn_Surfacegroupsinuniformlattices} their proof does not work in setting we are dealing with, namely split real Lie groups.\\

\noindent
\textbf{Organization of the paper.} In Section 1 we give some background on $\mathds{Q}$-arithmetic subgroups in semi-simple Lie groups and describe their classification for the relevant Lie groups. In Section 2 we consider the relations between the irreducible representation $\tau_n$ and the $\mathds{R}/\mathds{Q}$-forms of $\SL_n$. We prove there Propositions $\A$. In Section 3 we give the corresponding results for $\Sp(2n,\mathds{R})$, $\SO(\I_{k+1,k},\mathds{R})$ and $\textbf{G}_2(\mathds{R})$. In Section 4 we describe the ``bending" construction. Finally in Section 5 we give the proof of Theorems $\A$.\\

\noindent
\textbf{Acknowledgement.} The author would like to say that this paper would have not been possible without the monitoring of Andrés Sambarino. The author is especially grateful for all the enlightening and entertaining discussions and the numerous yet careful reading of this article. He would also like to thank Xenia Flamm for her remarks which made this article more understandable and Maxime Wolff for the time he spent discussing with us. The author is indebted to Marc Burger who recommended to look at his talk \cite{Burger_OnIntegerPointHitchin} from 2015. This talk inspired our proof on the existence of infinitely many mapping class group-orbits. Finally, the author thanks the referee for numerous helpful comments.

\section{Background on $\mathds{Q}$-arithmetic groups}
\label{sectionarithmeticgroup}

Let $k$ be an algebraic number field with ring of integers $\mathcal{O}_k$. A \emph{central simple algebra} over $k$ is a unital associative algebra $D$ which is simple and whose center is exactly $k$. Quaternion algebras are a particular case of central simple algebras. If $D\otimes_{k}\mathds{R}\simeq\M_d(\mathds{R})$, then the isomorphism $\phi:\M_n(D)\otimes_{k}\mathds{R}\xrightarrow{\sim}\M_{dn}(\mathds{R})$ induces a norm on $\M_n(D)$ given by
$\Nred:\M_n(D)\to k,\ \M\mapsto\Det(\phi(\M)).$
This norm is independent of $\phi$. We can thus define $$\SL(n,D)=\{\M\in\M_n(D)|\Nred(\M)=1\}.$$ If $\B\in\M_n(D)$ and $\partial:D\rightarrow D$ is an involution we define \begin{equation*}
    \SU(\B,\partial;D)=\{\M\in\SL(n,D)|\partial(\M^\top)\B\M=\B\}.
\end{equation*} Given a subring $R$ of $D$, let $\SL(n,R)=\SL(n,D)\cap\M_n(R)$ and $\SU(\B,\partial;R)=\SU(\B,\partial;D)\cap\M_n(R)$. These notations will be used to define $\mathds{Q}$-arithmetic subgroups of $\SL(n,\mathds{R})$.

An \emph{anti-involution} of $D$ is a map $\partial:D\rightarrow D$ such that for all $a,b\in D$ $$\partial(a+b)=\partial(a)+\partial(b),\ \partial(ab)=\partial(b)\partial(a)\ \text{and}\ \partial^2(a)=a.$$ For example, the conjugation $\overline{\phantom{s}}$ of a quaternion algebra is an anti-involution. An anti-involution induces an involution of $k$ that can be trivial or not. If $A$ is a quaternion algebra over $k$, with $L$ a quadratic extension of $k$ and $\sigma$ the non-trivial element of $\Gal(L/k)$, then $\overline{\phantom{s}}\otimes\sigma$ is an anti-involution on $A\otimes_{k}L$ which takes $x\otimes l$ to $\overline{x}\otimes\sigma(l)$ and thus induces a non-trivial involution of $L$.

We say that $D$ is a \emph{division algebra} if every non-zero element of $D$ is invertible. The \emph{degree} of a division algebra is the square root of its dimension. An \emph{order} of $D$ is a finitely generated $\mathcal{O}_{k}$-submodule of $D$ containing a basis of $D$ over $k$, which is a subring of $D$ and contains the unit element.

Let $\partial$ be an anti-involution on $D$ and let $f:D^n\times D^n\to D$ be a map satisfying
\begin{alignat*}{2}
f(\lambda u+v,w)&=\lambda f(u,w)+f(v,w)\\
f(u,\lambda v+w)&=\partial(\lambda)f(u,v)+f(u,w)\\ f(u,v)&=\partial(f(v,u))
\end{alignat*} for all $u,v,w\in D^n$ and $\lambda\in D$. Then $x\mapsto f(x,x)$ is called a \emph{$\partial$-Hermitian form}. If $\partial$ is trivial, we call it a \emph{quadratic form}. The \emph{discriminant} of a $\partial$-Hermitian form is the determinant of a matrix representing $f$ in any basis. It is well-defined up elements of the form $\lambda\partial(\lambda),\ \lambda\in D$. The \emph{rank} of a form is the rank of a matrix representing $f$ in any basis. The form is said to be \emph{non-degenerate} if it has full rank. Here are the results we will need about $\partial$-Hermitian forms.

\begin{proposition}[See \S4 in Lewis \cite{Lewis_IsometryClassificationHermitianForms} or example 5 of \S1 in Milnor \cite{Milnor_IsometriesInnerProductSpaces}]
\label{propHermitianQd}
Let $d\in\mathds{N}$ be not a square and let $\sigma\in\Gal(\mathds{Q}(\sqrt{d})/\mathds{Q})$ be non-trivial. Then $\sigma$-Hermitian forms on $\mathds{Q}(\sqrt{d})$ are classified by their rank and their discriminant\footnote{since $d>0$ we have that $\mathds{Q}(\sqrt{d})\otimes_{\mathds{Q}}\mathds{R}\simeq\mathds{R}\oplus\mathds{R}$ and thus there is no local invariant at the infinite place of $\mathds{Q}$}.
\end{proposition}

\begin{proposition}[See \S5 in \cite{Lewis_IsometryClassificationHermitianForms}]
\label{propHermitianconj}
Let $A$ be a quaternion division algebra over $\mathds{Q}$ that splits over $\mathds{R}$ with conjugation $\overline{\phantom{s}}$. Then $\overline{\phantom{s}}$-Hermitian forms on $A$ are classified by their rank.
\end{proposition}

We can state the classification of $\mathds{Q}$-arithmetic groups in the cases of $\SL(n,\mathds{R})$, $\Sp(2n,\mathds{R})$ and $\SO(\I_{k+1,k},\mathds{R})$. It is based on the classification of arithmetic subgroups in classical Lie groups, see Morris \S18.5 \cite{Morris_IntroductionArithmeticGroups}. See also Remark \ref{remarkG2lattice} for the classification of $\mathds{Q}$-arithmetic subgroups in the case of $\G_2$.

The $\mathds{Q}$-arithmetic subgroups of $\SL(n,\mathds{R})$ are the subgroups commensurable with conjugates of:
\begin{itemize}
    \setlength\itemsep{0cm}
    \item
    $\SL(m,\mathcal{O})$ where $\mathcal{O}$ is an order of a division algebra $D$ over $\mathds{Q}$ of degree $d$ such that $D\otimes_{\mathds{Q}}\mathds{R}\simeq\M_d(\mathds{R})$ and with $n=dm$,
    \item
    $\SU(\B,\partial;\mathcal{O})$ where $L$ is a quadratic real extension of $\mathds{Q}$, $D$ is a division algebra over $L$ of degree $d$ such that $D\otimes_{\mathds{Q}}\mathds{R}\simeq\M_d(\mathds{R})\oplus\M_d(\mathds{R})$, $\mathcal{O}$ is an order in $D$, $\partial$ is an anti-involution of $D$ whose restriction to $L$ is the non-trivial automorphism of $L$ over $\mathds{Q}$, $\B$ is an invertible matrix in $\M_m(D)$ satisfying $\partial(\B^\top)=\B$ with $n=md$.
\end{itemize}

The following proposition will be used in the proof of Corollary \ref{corollary1}.

\begin{proposition}
\label{propositionprime}
Let $p\neq2$ be prime. The non-uniform lattices of $\SL(p,\mathds{R})$ are the subgroups commensurable with conjugates of $\SL(p,\mathds{Z})$ or of  $\SU(\I_p,\sigma;\mathds{Z}[\sqrt{d}])$ for $d\in\mathds{N}$ not a square and for the non-trivial $\sigma\in\Gal(\mathds{Q}(\sqrt{d})/\mathds{Q})$.
\end{proposition}

\begin{proof}
By Margulis Arithmeticity Theorem (Theorem 16.3.1 in \cite{Morris_IntroductionArithmeticGroups}) all non-uniform lattices of $\SL(p,\mathds{R})$ are arithmetic subgroups. As can be seen in \cite{Morris_IntroductionArithmeticGroups}, Corollary 5.3.2, non-uniform arithmetic subgroups are $\mathds{Q}$-arithmetic. Thus they are commensurable with conjugates of the $\mathds{Z}$-points of a $\mathds{R}/\mathds{Q}$-form of $\SL_p$. As shown in \S18.5 of \cite{Morris_IntroductionArithmeticGroups}, their $\mathds{Q}$-points can only be of two kinds.

First they can be of the form $\SL(m,D)$ for $D$ a division algebra of degree $r$ over $\mathds{Q}$ such that $D\otimes_{\mathds{Q}}\mathds{R}\simeq\M_r(\mathds{R})$. Here we must have $rm=p$. If $(r,m)=(1,p)$, the corresponding arithmetic subgroup is of the form $\SL(p,\mathds{Z})$. If $(r,m)=(p,1)$, it is of the form $\SL(1,\mathcal{O})$ for $\mathcal{O}$ an order of $D$, but this is a uniform lattice (see the classification in \S18.5 in \cite{Morris_IntroductionArithmeticGroups}).

Secondly they can be of the form $\SU(\B,\partial;D)$ for $D$ a division algebra of degree $r$ over a real quadratic extension $L$ of $\mathds{Q}$ such that $D\otimes_{\mathds{Q}}L\simeq\M_r(\mathds{R})\oplus\M_r(\mathds{R})$, $\partial$ an anti-involution of $D$ whose restriction to $L$ is the non-trivial $\sigma\in\Gal(L/\mathds{Q})$ and $\B\in\M_m(D)$ invertible satisfying $\partial(\B^\top)=\B$. Once again we must have $mr=p$. If $(r,m)=(1,p)$, the corresponding arithmetic subgroup is of the form $\SU(\B,\sigma;\mathds{Z}[\sqrt{d}])$ where $L=\mathds{Q}(\sqrt{d})$. If $(r,m)=(p,1)$, it is of the form $\SU(\B,\partial;\mathcal{O})$ where $\mathcal{O}$ is an order of $D$, but this is a uniform lattice (see the classification in \S18.5 in \cite{Morris_IntroductionArithmeticGroups}).

It remains to show that all subgroups of the form $\SU(\B,\sigma;\mathds{Z}[\sqrt{d}])$ are conjugates of $\SU(\I_n,\sigma;\mathds{Z}[\sqrt{d}])$. By Proposition \ref{propHermitianQd}, up to scalar multiplication and conjugation there is only one non-degenerate Hermitian form on $\mathds{Q}(\sqrt{d})^p$. Hence $\SU(\B,\sigma;\mathds{Z}[\sqrt{d}])$ is conjugate to $\SU(\I_n,\sigma;\mathds{Z}[\sqrt{d}])$.
\end{proof}

\begin{proposition}
The $\mathds{Q}$-arithmetic subgroups of $\SO(\I_{k+1,k},\mathds{R})$, $k\geq 2$, are the subgroups commensurable with conjugates of $\SO(\B,\mathds{Z})$ for $\B\in\SL(2k+1,\mathds{Q})$ symmetric of signature $(k+1,k)$ if $k$ is even and $(k,k+1)$ if $k$ is odd. 
\end{proposition}

\begin{proof}
This follows form the classification in \S18.5 in \cite{Morris_IntroductionArithmeticGroups}. Note that we can always assume that $\B$ has determinant 1 up to multiplying $\B$ by a scalar and up to congruence.
\end{proof}

The $\mathds{Q}$-arithmetic subgroups of $\Sp(2n,\mathds{R})$ are the subgroups commensurable with conjugates of:
\begin{itemize}
    \setlength\itemsep{0cm}
    \item $\Sp(2n,\mathds{Z})$
    \item $\SU(\B,\overline{\phantom{s}};\mathcal{O})$ where $A$ is a division quaternion algebra over $\mathds{Q}$ such that $A\otimes_{\mathds{Q}}\mathds{R}\simeq\M_2(\mathds{R})$, $\mathcal{O}$ is an order in $A$, $\overline{\phantom{s}}$ is the conjugation of $A$ and $\B$ is an invertible matrix in $\M_n(A)$ satisfying $\overline{\B}^\top=\B$.
\end{itemize}

\begin{proposition}
\label{propositionsymplectic}
Let $n\geq2$. Non-uniform lattices of $\Sp(2n,\mathds{R})$ are subgroups commensurable with conjugates of $\Sp(2n,\mathds{Z})$ or of $\SU(\I_n,\overline{\phantom{s}};\mathcal{O})$ for $\mathcal{O}$ an order of a division quaternion algebra over $\mathds{Q}$ that splits over $\mathds{R}$ with conjugation $\overline{\phantom{s}}$.
\end{proposition}

\begin{proof}
By Margulis Arithmeticity Theorem (Theorem 16.3.1 in \cite{Morris_IntroductionArithmeticGroups}) all non-uniform lattices of $\Sp(2n,\mathds{R})$ are arithmetic subgroups. As can be seen in \cite{Morris_IntroductionArithmeticGroups}, Corollary 5.3.2, all non-uniform arithmetic groups are $\mathds{Q}$-arithmetic. From the classification these can only be commensurable to conjugates of $\Sp(2n,\mathds{Z})$ or of $\SU(\B,\overline{\phantom{s}};\mathcal{O})$ with notations as above. By Proposition \ref{propHermitianconj}, there is only one nondegenerate $\overline{\phantom{s}}$-Hermitian form on a division quaternion algebra over $\mathds{Q}$ that splits over $\mathds{R}$. Thus $\SU(\B,\overline{\phantom{s}};\mathcal{O})$ is conjugate to $\SU(\I_n,\overline{\phantom{s}};\mathcal{O})$.
\end{proof}

\section{$\R/\QQ$-forms of $\SL_n$}
\label{sectionQformsSL}

In this section we study $\mathds{R}/\mathds{Q}$-forms of $\SL_n$ using cohomological tools. 

\subsection{Non-abelian cohomology}

Let $G$ be a topological group acting continuously on a topological group $H$ (not necessarily abelian) by automorphisms. For $g\in G$ and $\phi\in\Aut(H)$, define
$$\prescript{g}{}{\phi}:x\mapsto g\phi(g^{-1}x).$$
A \emph{$1$-cocycle} is a map $\xi:G\to\Aut(H)$ that satisfies
$$\xi(g_1g_2)=\xi(g_1)\circ\prescript{g_1}{}{\xi(g_2)}$$ for all $g_1,g_2\in G$. We say that two $1$-cocycles $\xi$ and $\zeta$ are \emph{equivalent} if there exists $\phi\in\Aut(H)$ such that for all $g\in G$ we have
$$\xi(g)=\phi^{-1}\circ\zeta(g)\circ\prescript{g}{}{\phi}.$$

\begin{definition}
The \emph{first Galois cohomology set of $G$ (with coefficient in $H$)} is the set of equivalence classes of continuous $1$-cocycles. It is denoted by $\HH^1(G,H)$. This is not a group in general.
\end{definition}

\subsection{$\R/\QQ$-forms and nonabelian Galois cohomology}

We want to describe $\mathds{R}/\mathds{Q}$-forms of $\SL_n$. Unfortunately, the extension $\mathds{R}/\mathds{Q}$ is not Galois, which prevents us from using the cohomological description of forms of algebraic groups. However, thanks to Proposition 2 of III.\S1 in \cite{Serre_GaloisCohomology},
$\mathds{R}/\mathds{Q}$-forms of $\SL_n$ are $\overline{\mathds{Q}}/\mathds{Q}$-forms of $\SL_n$ where $\overline{\mathds{Q}}$ is the subfield of $\mathds{C}$ of algebraic numbers over $\mathds{Q}$.

Let $k$ be an algebraic number field. We start by recalling the cohomological description of $\overline{\mathds{Q}}/k$-forms of an algebraic group. The main references for the material covered here are \cite{Morris_IntroductionArithmeticGroups} chapter 18 and Platonov and Rapinchuk's book \cite{Platonov_AlgebraicgroupsNumbertheory} section 1.3.

Let $\G$ be a $k$-algebraic subgroup of $\GL_n$. A \emph{$\overline{\mathds{Q}}/k$-form} of $\G$ is an algebraic group $\HH$ over $k$ such that $\HH\simeq\G$ over $\overline{\mathds{Q}}$. We endow $\Gal(\overline{\mathds{Q}}/k)$ with its profinite topology (see \cite{Platonov_AlgebraicgroupsNumbertheory} page 22) and $\Aut(\G(\overline{\mathds{Q}}))$ with the discrete topology. An element $\sigma\in\Gal(\overline{\mathds{Q}}/k)$ acts on $\G(\overline{\mathds{Q}})$ by applying $\sigma$ to all entries of the matrix $g\in\G(\overline{\mathds{Q}})$. This allows us to define $\HH^1(\Gal(\overline{\mathds{Q}}/k),\Aut(\G(\overline{\mathds{Q}})))$, the first Galois cohomology set of $\G$.

Let $\xi:\Gal(\overline{\mathds{Q}}/k))\rightarrow \Aut(\G(\overline{\mathds{Q}}))$ be a continuous 1-cocycle and denote by $\overline{\mathds{Q}}[\G]$ the coordinate ring of $\G$. We define
$$\prescript{}{\xi}{k[\G]}=\{f\in\overline{\mathds{Q}}[\G]\ |\ \sigma\circ f\circ\sigma^{-1}\circ\xi(\sigma)^{-1}=f\ \forall\sigma\in\Gal(\overline{\mathds{Q}}/k)\}.$$ This is a $k$-algebra and we let $\prescript{}{\xi}{\G}$ be the $k$-algebraic subgroup of $\GL_n$ that has $\prescript{}{\xi}{k[\G]}$ as coordinate ring. It is a $\overline{\mathds{Q}}/k$-form of $\G$ since $\overline{\mathds{Q}}\otimes_{k}\prescript{}{\xi}{k[\G]}\simeq\overline{\mathds{Q}}[\G]$. Concretely, its $k$-points are
$$\prescript{}{\xi}{\G}(k)=\{\M\in\G(\overline{\mathds{Q}})\ |\ \xi(\sigma)\circ\sigma(\M)=\M\ \forall\sigma\in\Gal(\overline{\mathds{Q}}/k)\}.$$

Conversely if $\HH$ is a $\overline{\mathds{Q}}/k$-form of $\G$ and $\Phi:\HH(\overline{\mathds{Q}})\rightarrow\G(\overline{\mathds{Q}})$ is an isomorphism, let
\begin{equation*}
    \xi:\Gal(\overline{\mathds{Q}}/k)\rightarrow\Aut(\G(\overline{\mathds{Q}})), \sigma\mapsto\Phi\circ\sigma\circ\Phi^{-1}\circ\sigma^{-1}.
\end{equation*}
This is a continuous 1-cocycle. In this way, $\HH^1(\Gal(\overline{\mathds{Q}}/k),\Aut(\G(\overline{\mathds{Q}})))$ classifies $\overline{\mathds{Q}}/k$-forms of $\G$ (see \cite{Platonov_AlgebraicgroupsNumbertheory} Theorem 2.9 and \S2.2.3):

\begin{proposition}
Denote by $\mathcal{F}$ the set of isomorphism classes of $\overline{\mathds{Q}}/k$-forms of $\G$. The map $$\HH^1(\Gal(\overline{\mathds{Q}}/k),\Aut(\G(\overline{\mathds{Q}})))\to\mathcal{F},\ [\xi]\mapsto[\prescript{}{\xi}{\G}]$$ is a bijection.
\end{proposition}

A $\overline{\mathds{Q}}/k$-form of $\G$ is said to be \emph{inner} if its associated 1-cocycle has values in $\Inn(\G(\overline{\mathds{Q}}))$, the group of inner automorphisms. Since $\Aut(\SL(2,\overline{\mathds{Q}}))=\Inn(\SL(2,\overline{\mathds{Q}}))$, every $\overline{\mathds{Q}}/k$-form of $\SL_2$ is inner.

\subsection{Compatible cocycles}

Let $a,b\in\mathds{N}$. Define
\begin{align*}
    &\T^{a,b}:\Gal(\overline{\mathds{Q}}/\mathds{Q})\rightarrow\PGL(2,\overline{\mathds{Q}})=\PSL(2,\overline{\mathds{Q}})\\
    \sigma&\mapsto\left\{\begin{array}{ll}
        \I_2 &\mbox{if $\sigma(\sqrt{a})=\sqrt{a}$ and $\sigma(\sqrt{b})=\sqrt{b}$}  \\
        \begin{pmatrix}
            1&0\\
            0&-1
        \end{pmatrix} &\mbox{if $\sigma(\sqrt{a})=\sqrt{a}$ and $\sigma(\sqrt{b})=-\sqrt{b}$}\\
        \begin{pmatrix}
        0&1\\
        1&0
        \end{pmatrix} &\mbox{if $\sigma(\sqrt{a})=-\sqrt{a}$ and $\sigma(\sqrt{b})=\sqrt{b}$}\\
        \begin{pmatrix}
        0&1\\
        -1&0
        \end{pmatrix} &\mbox{if $\sigma(\sqrt{a})=-\sqrt{a}$ and $\sigma(\sqrt{b})=-\sqrt{b}$.}
    \end{array}
    \right.
\end{align*}
For simplicity we will often write $\T_{\sigma}$ or $\T^{a,b}_{\sigma}$ instead of $\T^{a,b}({\sigma})$. Note that \begin{align*}
    &\tau_n(\T^{a,b}):\Gal(\overline{\mathds{Q}}/\mathds{Q})\rightarrow\PSL(n,\overline{\mathds{Q}})\\
    \sigma\mapsto&\left\{\begin{array}{ll}
        \I_n &\mbox{if $\sigma(\sqrt{a})=\sqrt{a}$ and $\sigma(\sqrt{b})=\sqrt{b}$}  \\
        \begin{pmatrix}
            (-1)^{n-1}&&\\
            &\ddots&\\
            &&1
        \end{pmatrix} &\mbox{if $\sigma(\sqrt{a})=\sqrt{a}$ and $\sigma(\sqrt{b})=-\sqrt{b}$}\\
        \begin{pmatrix}
        &&1\\
        &\iddots&\\
        1&&
        \end{pmatrix} &\mbox{if $\sigma(\sqrt{a})=-\sqrt{a}$ and $\sigma(\sqrt{b})=\sqrt{b}$}\\
        \begin{pmatrix}
        &&1\\
        &\iddots&\\
        (-1)^{n-1}&&
        \end{pmatrix} &\mbox{if $\sigma(\sqrt{a})=-\sqrt{a}$ and $\sigma(\sqrt{b})=-\sqrt{b}$.}
    \end{array}
    \right.
\end{align*}

For an invertible matrix $\M$, denote by $\Int(\M)$ the automorphism of $\SL_n(\overline{\mathds{Q}})$ given by $\X\mapsto \M\X\M^{-1}$.

\begin{lemma}
\label{lemma1}
For a continuous $1$-cocycle $\xi:\Gal(\overline{\mathds{Q}}/\mathds{Q})\rightarrow\Aut(\SL_2(\overline{\mathds{Q}}))$ such that $\prescript{}{\xi}{\SL_2(\mathds{R})}\simeq\SL(2,\mathds{R})$, there exist $a,b\in\mathds{N}$ and $\PP\in\SL(2,\overline{\mathds{Q}})$ such that $\xi(\sigma)=\Int(\PP^{-1}\T^{a,b}_{\sigma}\sigma(\PP))$ for every $\sigma\in\Gal(\overline{\mathds{Q}}/\mathds{Q})$.
\end{lemma}

\begin{proof}
Every $\overline{\mathds{Q}}/\mathds{Q}$-form of $\SL_2$ is of the form $\HH(K)=\SL(1,A\otimes_{\mathds{Q}}K)$ ($K$ any field extension of $\mathds{Q}$) with $A=(a,b)_{\mathds{Q}}$ a quaternion algebra (see Proposition 2.17 in \cite{Platonov_AlgebraicgroupsNumbertheory}). If $a$ and $b$ are both negative then $\prescript{}{\xi}{\SL_2(\mathds{R})}\simeq\SU(2)$. Thus $a$ or $b$ is positive and up to a change of basis we can assume that $a$ and $b$ are both positive integers. Consider the group
\begin{equation*}
    \SL(1,A)=\begin{Bmatrix}
    \begin{pmatrix}
    x_0+\sqrt{a}x_1&\sqrt{b}x_2+\sqrt{ab}x_3\\
    \sqrt{b}x_2-\sqrt{ab}x_3&x_0-\sqrt{a}x_1
    \end{pmatrix}
    |\ x_i\in\mathds{Q},\ \Det=1\end{Bmatrix}
\end{equation*} consisting of the elements of $A$ of norm $1$.
When the $x_i$ are taken in $\overline{\mathds{Q}}$ we actually get $\SL(2,\overline{\mathds{Q}})$. Denote by $\Phi$ the identity embedding in $\SL(2,\overline{\mathds{Q}})$.
Direct computation shows that the 1-cocycle associated to $\Phi$ is $\sigma\mapsto\Int(\T^{a,b}_{\sigma})$. The $\overline{\mathds{Q}}/\mathds{Q}$-form of $\SL_2$ associated to $\xi$ is a conjugate of $
\SL(1,\A)$ whence the result.
\end{proof}

\begin{definition}
\label{definitioncompatible}
Let $n\geq3$ and let $\G$ be a $\mathds{Q}$-algebraic subgroup of $\SL_n$ such that $\tau_n(\SL(2,\overline{\mathds{Q}}))<\G(\overline{\mathds{Q}})$. Let $\xi:\Gal(\overline{\mathds{Q}}/\mathds{Q})\to\Aut(\SL_2(\overline{\mathds{Q}}))$ and $\zeta:\Gal(\overline{\mathds{Q}}/\mathds{Q})\to\Aut(\G(\overline{\mathds{Q}}))$ be $1$-cocycles. We say that $\zeta$ is \emph{$\tau_n$-compatible} with $\xi$ if $\tau_n(\prescript{}{\xi}{\SL_2(\mathds{Q})})<\prescript{}{\zeta}{\G(\mathds{Q})}$.
\end{definition}

\begin{definition}
\label{definition1}
For all $n\geq2$, let $\J_n$ be
\begin{equation*}
\begin{pNiceMatrix}
    &&&&(n-1)!\\
    &&&\Iddots&\\
    &&(-1)^{i-1}(n-i)!(i-1)!&&\\
    &\Iddots&&&\\
    (-1)^{n-1}(n-1)!&&&&&\\
\end{pNiceMatrix}.
\end{equation*}
Note that for any $\M\in\SL(2,\mathds{C})$, $\tau_n(\M)^\top\J_n\tau_n(\M)=\J_n$ (see for example McGarraghy \cite{McGarraghy_SymmetricPowersSymmetricBilinear}).
\end{definition}

\begin{proposition}
\label{proposition1}
Let $n\geq 3$. Let $\xi:\Gal(\overline{\mathds{Q}}/\mathds{Q})\rightarrow\Aut(\SL_2(\overline{\mathds{Q}}))$ and $\zeta:\Gal(\overline{\mathds{Q}}/\mathds{Q})\to\Aut(\SL_n(\overline{\mathds{Q}}))$ be continuous $1$-cocycles. Suppose that $\prescript{}{\xi}{\SL_2(\mathds{R})}\simeq\SL(2,\mathds{R})$. Write $\xi(\sigma)=\Int(\PP^{-1}\T^{a,b}_{\sigma}\sigma(\PP))$ for all $\sigma$ with $\PP\in\PSL(2,\overline{\mathds{Q}})$. Then $\zeta$ is $\tau_n$-compatible with $\xi$ if and only if either
\begin{equation}
\label{equation1}
\zeta:\sigma\mapsto\Int(\tau_n(\PP^{-1}\T^{a,b}_{\sigma}\sigma(\PP)))
\end{equation}
or there exists a quadratic field extension $\mathds{Q}(\sqrt{d})$ of $\mathds{Q}$ such that
\begin{equation}
\label{equation2}
    \zeta:\sigma\mapsto\left\{
    \begin{array}{ll}
        \Int(\tau_n(\PP^{-1}\T^{a,b}_{\sigma}\sigma(\PP))) & \mbox{if $\sigma(\sqrt{d})=\sqrt{d}$} \\
        \Int(\tau_n(\PP^{-1}\T^{a,b}_{\sigma})\J_n^{-1})\circ\omega\circ\Int(\tau_n\circ\sigma(\PP)) & \mbox{if $\sigma(\sqrt{d})=-\sqrt{d}$}
    \end{array}\right.
\end{equation}
where $\omega(\M)=(\M^\top)^{-1}$.
\end{proposition}

\begin{proof}
Up to conjugation we can assume that $\PP=\I_2$.
Suppose that $\tau_n({}_{\xi}\SL_2(\mathds{Q}))<{}_{\zeta}\SL_n(\mathds{Q})$. Fix from now on $\sigma\in\Gal(\overline{\mathds{Q}}/\mathds{Q})$. It follows from $\Aut(\SL_n(\overline{\mathds{Q}}))=\Inn(\SL_n(\overline{\mathds{Q}}))\rtimes\langle\omega\rangle$ (see Theorem 2.8 in \cite{Platonov_AlgebraicgroupsNumbertheory}) that $\zeta(\sigma)$ is either inner or the composition of an inner automorphism with $\omega$.

Suppose that $\zeta(\sigma)=\Int(\A_{\sigma})$ for some $\A_\sigma\in\PSL(n,\overline{\mathds{Q}})$. Then, by definition, all $\M\in{}_{\xi}\SL_2(\mathds{Q})$ satisfy
\begin{align*}
    &\tau_n(\T_{\sigma}\sigma(\M)\T_{\sigma}^{-1})=\tau_n(\M)=\A_{\sigma}\sigma(\tau_n(\M))\A_{\sigma}^{-1}\\
    \implies&\A_{\sigma}\tau_n(\T_{\sigma}^{-1})\tau_n(\M)=\tau_n(\M)\A_{\sigma}\tau_n(\T_{\sigma})^{-1}
\end{align*}
since $\tau_n\circ\sigma=\sigma\circ\tau_n$. Hence $\A_{\sigma}\tau_n(\T_{\sigma})^{-1}$ commutes with $\tau_n(\prescript{}{\xi}{\SL_2(\mathds{Q})})$ thus with its Zariski-closure which is $\tau_n(\SL_2(\overline{\mathds{Q}}))$. Schur's lemma (see Lemma 1.7 in \cite{Fulton_RepresentationTheory}) shows that $\A_{\sigma}=\tau_n(\T_{\sigma})$. It follows that if $\zeta$ is inner, it is of the form of equation \eqref{equation1}.

Suppose that $\zeta(\sigma)=\Int(\A_{\sigma})\circ\omega$ for some $\A_{\sigma}\in\PSL(n,\overline{\mathds{Q}})$. Then for all $\M\in{}_{\xi}\SL_2(\mathds{Q})$
\begin{align*}
    &\tau_n(\T_{\sigma}\sigma(\M)\T_{\sigma}^{-1})=\tau_n(\M)=\A_{\sigma}\omega(\sigma\circ\tau_n(\M))\A_{\sigma}^{-1}\\
    \implies&\tau_n(\T_{\sigma}\sigma(\M)\T_{\sigma}^{-1})=\tau_n(\M)=\A_{\sigma}\J_n\sigma(\tau_n(\M))\J_n^{-1}\A_{\sigma}^{-1}\\
    \implies&\A_{\sigma}\J_n\tau_n(\T_{\sigma})^{-1}\tau_n(\M)=\tau_n(\M)\A_{\sigma}\J_n\tau_n(\T_{\sigma})^{-1}.
\end{align*}
By the argument above we get $\A_{\sigma}=\tau_n(\T_{\sigma})\J_n^{-1}$.

Since $\J_n$ is in $\PGL(n,\mathds{Q})$ and $\T_{\sigma}$ has a representative in $\PSL(2,\mathds{Q})$, for all $\sigma$ the automorphism $\zeta(\sigma)$ is fixed under the action of $\Gal(\overline{\mathds{Q}}/\mathds{Q})$. Applying the $1$-cocycle formula we see that $\zeta(\sigma\tau)=\zeta(\sigma)\zeta(\tau)$ for all $\sigma,\tau\in\Gal(\overline{\mathds{Q}}/\mathds{Q})$ which means that $\zeta$ is group homomorphism. Suppose $\zeta$ is not inner. Then $\W=\zeta^{-1}(\Inn(\SL_n(\overline{\mathds{Q}})))$ is of index 2 in $\Gal(\overline{\mathds{Q}}/\mathds{Q})$. It is closed and open because $\zeta$ is continuous. Thus its fixed field is a quadratic extension of $\mathds{Q}$ which can be written as $\mathds{Q}(\sqrt{d})$ for some $d\in\mathds{Z}$. Finally we have shown that $\zeta$ is of the form of equation \eqref{equation2}.

Conversely, all maps defined by \eqref{equation1} or \eqref{equation2} are 1-cocycles as shown by computation using the fact that $\tau_n(\T_{\sigma})$ and $\J_n$ commute.
\end{proof}

\subsection{Determining the $\QQ$-form associated to a $\tau_n$-compatible cocycle}

We now explain how to determine the $\overline{\mathds{Q}}/\mathds{Q}$-form $\prescript{}{\zeta}{\SL_n}$ for a $1$-cocycle $\zeta$ which $\tau_n$-compatible with $\xi$.

Let $k$ be a field. Wedderburn's Theorem (see \cite{Maclachlan_ArithmeticHyperbolic3Manifolds} Theorem 2.9.6) states that every central simple algebra over $k$ is isomorphic to $\M_n(D)$ for some $n$ and some division algebra $D$. We say that two central simple algebras $A_1$ and $A_2$ over $k$ are \emph{equivalent} if $A_1\simeq\M_n(D)$ and $A_2\simeq\M_m(D)$ for a division algebra $D$ over $k$. The set of equivalence classes of central simple algebras over $k$ is called the \emph{Brauer group} and is denoted by $\Br(k)$. It has a group structure induced by the tensor product, the inverse of a given central simple algebra being its opposite algebra.

Suppose that $k$ is an algebraic number field and consider $\overline{\mathds{Q}}^*$ endowed with the discrete topology. A \emph{factor set} (also called a \emph{$2$-cocycle}) is a map $f:\Gal(\overline{\mathds{Q}}/k)\times\Gal(\overline{\mathds{Q}}/k)\to\overline{\mathds{Q}}^*$ that satisfies
$$\sigma_1f(\sigma_2,\sigma_3)f(\sigma_1\sigma_2,\sigma_3)^{-1}f(\sigma_1,\sigma_2\sigma_3)f(\sigma_1,\sigma_2)^{-1}=1$$
for any $\sigma_1,\sigma_2,\sigma_3\in\Gal(\overline{\mathds{Q}}/k)$. We will often write $f_{\sigma,\tau}$ for $f(\sigma,\tau)$. The abelian group structure of $\overline{\mathds{Q}}^*$ induces a group structure on the set of factor sets.
A factor set $f$ is said to be \emph{trivial} if there exists $\phi:\Gal(\overline{\mathds{Q}}/k)\to\overline{\mathds{Q}}^*$ such that
$$f(\sigma_1,\sigma_2)=\phi(\sigma_1\sigma_2)\phi(\sigma_1)^{-1}\sigma_1(\phi(\sigma_2))^{-1}$$ for all $\sigma_1,\sigma_2\in\Gal(\overline{\mathds{Q}}/k)$. We denote by $\HH^2(\Gal(\overline{\mathds{Q}}/k),\overline{\mathds{Q}}^*)$ the group of continuous factor sets modulo the subgroup consisting of trivial factor sets. Since $\overline{\mathds{Q}}^*$ is abelian, this is usual group cohomology.

From an equivalence class $[(a_{\sigma,\tau})_{\sigma,\tau}]\in\HH^2(\Gal(\overline{\mathds{Q}}/k),\overline{\mathds{Q}}^*)$, we construct an associated central simple algebra as follows. Since $$\HH^2(\Gal(\overline{\mathds{Q}}/k),\overline{\mathds{Q}}^*)=\varinjlim \HH^2(\Gal(K/k),K^*)$$ where $K$ runs through all finite Galois extensions of $k$, there exist such a $K$ and a class of factor set $[(b_{\sigma,\tau})_{\sigma,\tau}]\in \HH^2(\Gal(K/k),K^*)$ such that $[(a_{\sigma,\tau})_{\sigma,\tau}]$ is the image of $[(b_{\sigma,\tau})_{\sigma,\tau}]$. We then consider $(K,k,(b_{\sigma,\tau})_{\sigma,\tau})$ the central simple algebra over $k$ whose underlying vector space is
$$\bigoplus\limits_{\sigma\in\Gal(K/k)}K v_{\sigma}$$
and multiplication is defined by
\begin{equation*}
    (xv_\sigma)(yv_\tau)=x\sigma(y)b_{\sigma,\tau}v_{\sigma\tau}\ \textrm{for all $\sigma,\tau\in\Gal(K/k)$}.
\end{equation*} The equivalence class of central simple algebras corresponding to $[(a_{\sigma,\tau})_{\sigma,\tau}]$ in the Brauer group is the class of $(K,k,(b_{\sigma,\tau})_{\sigma,\tau})$. It follows that (see Theorem 4.4.7 in \cite{Gille_CentralSimpleAlgebrasGaloisCohomology}):

\begin{proposition}
\label{propositionBrauer}
The map that associates an equivalence class of central simple algebras to a class of factor set induces an isomorphism $$\HH^2(\Gal(\overline{\mathds{Q}}/k),\overline{\mathds{Q}}^*)\simeq\Br(k).$$
\end{proposition}

The exact sequence
\begin{equation*}
    1\rightarrow\overline{\mathds{Q}}^*\rightarrow\GL_n(\overline{\mathds{Q}})\xrightarrow[]{\text{$\pi$}}\PSL_n(\overline{\mathds{Q}})\rightarrow1
\end{equation*}
gives rise to a map $$\delta:\HH^1(\Gal(\overline{\mathds{Q}}/k),\PSL_n(\overline{\mathds{Q}}))\rightarrow \HH^2(\Gal(\overline{\mathds{Q}}/k),\overline{\mathds{Q}}^*)\simeq\Br(k)$$
such that if $[\zeta]\in\HH^1(\Gal(\overline{\mathds{Q}}/k),\PSL_n(\overline{\mathds{Q}}))$ corresponds to the inner $\overline{\mathds{Q}}/k$-form $$\HH(K)=\SL(n,D\otimes_{k}K)$$ for any field extension $K$ of $k$ and with $D$ a central simple algebra over $k$, then $\delta(\zeta)=D$. We can calculate $\delta([\zeta])$ explicitly. First, for every $\sigma\in\Gal(\overline{\mathds{Q}}/k)$ choose an element $\M_{\sigma}\in\GL_n(\overline{\mathds{Q}})$ such that $\pi(\M_{\sigma})=\zeta(\sigma)$. Then define $a_{\sigma,\tau}=\M_\sigma\sigma(\M_\tau)\M_{\sigma\tau}^{-1}\in\overline{\mathds{Q}}^*$ for every $\sigma,\tau\in\Gal(\overline{\mathds{Q}}/k)$. This is a factor set. Under the isomorphism of Proposition \ref{propositionBrauer}, it corresponds to the equivalence class of $D$.

If $\xi:\Gal(\overline{\mathds{Q}}/\mathds{Q})\to\PSL(2,\overline{\mathds{Q}})$ is a cocycle we define
$$\prescript{}{\xi}{\M_2}(\mathds{Q})=\{x\in\M_2(\overline{\mathds{Q}})\ |\ \xi(\sigma)(\sigma(x))\xi(\sigma)^{-1}=x\}.$$
It is a quaternion algebra over $\mathds{Q}$.

\begin{proposition}
\label{proposition23}
Let $\xi:\Gal(\overline{\mathds{Q}}/\mathds{Q})\rightarrow\Aut(\SL_2(\overline{\mathds{Q}}))$ be a continuous $1$-cocycle such that $\prescript{}{\xi}{\SL_2(\mathds{R})}\simeq\SL(2,\mathds{R})$. Let $\zeta:\Gal(\overline{\mathds{Q}}/\mathds{Q})\rightarrow\Aut(\SL_n(\overline{\mathds{Q}}))$ be a continuous $1$-cocycle $\tau_n$-compatible with $\xi$ such that $\prescript{}{\zeta}{\SL_n}(\mathds{R})\simeq\SL(n,\mathds{R})$. If $\zeta$ is inner, then
\begin{equation*}
    \prescript{}{\zeta}{\SL_n(\mathds{Q})}\simeq\left\{
    \begin{array}{ll}
        \SL(n,\mathds{Q}) & \mbox{if n is odd} \\
        \SL(\frac{n}{2},\prescript{}{\xi}{\M_2(\mathds{Q})}) & \mbox{if n is even.}
    \end{array}\right.
\end{equation*}
If $\zeta$ is not inner, let $\mathds{Q}(\sqrt{d})$ be the associated quadratic extension. Then
\begin{equation*}
    {}_{\zeta}\SL_n(\mathds{Q})\simeq\left\{
    \begin{array}{ll}
        \SU(\I_n,\sigma;\mathds{Q}(\sqrt{d})) & \mbox{if $n$ is odd} \\
        \SU(\I_{\frac{n}{2}},\overline{\phantom{s}}\otimes\sigma;\prescript{}{\xi}{\M_2(\mathds{Q})}\otimes_{\mathds{Q}}\mathds{Q}(\sqrt{d})) & \mbox{if $n$ is even.}
    \end{array}\right.
\end{equation*}
\end{proposition}

\begin{proof}
Let $k$ denote either $\mathds{Q}$ or $\mathds{Q}(\sqrt{d})$ depending on whether $\zeta$ is inner or not. The restriction $\xi_{k}$ of $\xi$ to $\Gal(\overline{\mathds{Q}}/k)$ has values in $\PSL(2,\overline{\mathds{Q}})$. For every $\sigma\in\Gal(\overline{\mathds{Q}}/k)$ choose $\M_{\sigma}\in\GL_2(\overline{\mathds{Q}})$ such that $\pi(\M_{\sigma})=\xi_{k}(\sigma)$. Define $a_{\sigma,\tau}=\M_\sigma\sigma(\M_\tau)\M_{\sigma\tau}^{-1}$ for every $\sigma,\tau\in\Gal(\overline{\mathds{Q}}/k)$. Note that $\zeta_{k}(\sigma)=\tau_n(\xi_{k}(\sigma))$ for all $\sigma$. Then
\begin{equation*}
    \zeta_{k}(\sigma)=\tau_n(\xi_{k}(\sigma))=\tau_n(\pi(\M_{\sigma}))=\pi(\tau_n(\M_{\sigma}))
\end{equation*}
for every $\sigma\in\Gal(\overline{\mathds{Q}}/k)$. Thus $$\tau_n(\M_\sigma)\sigma(\tau_n(\M_\tau))\tau_n(\M_{\sigma\tau}^{-1})=\tau_n(\M_\sigma\sigma(\M_\tau)\M_{\sigma\tau}^{-1})$$ for every $\sigma,\tau\in\Gal(\overline{\mathds{Q}}/k)$, so that $\delta(\zeta)$ is the 2-cocycle given by $(a_{\sigma,\tau})^{n-1}$. Define
\begin{align*}
    \psi:\Gal(\overline{\mathds{Q}}/k)&\rightarrow\overline{\mathds{Q}}^*\\
    \sigma&\mapsto\Det(\M_{\sigma})
\end{align*}
and note that
\begin{align*}
    \psi(\sigma)\sigma(\psi(\tau))\psi(\sigma\tau)^{-1}&=\Det(\M_\sigma)\sigma(\Det(\M_\tau))\Det(\M_{\sigma\tau})^{-1}\\
    &=\Det(\M_\sigma\sigma(\M_\tau)\M_{\sigma\tau}^{-1})\\
    &=\Det(a_{\sigma,\tau}\I_2)=(a_{\sigma,\tau})^2
\end{align*}
for all $\sigma,\tau\in\Gal(\overline{\mathds{Q}}/k)$.
If $n$ is odd then $(a_{\sigma,\tau})^{n-1}=\psi(\sigma)^l\sigma(\psi(\tau))^l\psi(\sigma\tau)^{-l}$ where $l=\frac{n-1}{2}$, which shows that the class of $(a_{\sigma,\tau})_{\sigma,\tau}$ is actually trivial in $\HH^2(\Gal(\overline{\mathds{Q}}/k),\overline{\mathds{Q}}^*)$. If $n$ is even then $(a_{\sigma,\tau})^{n-1}=\psi(\sigma)^l\sigma(\psi(\tau))^l\psi(\sigma\tau)^{-l}a_{\sigma,\tau}$ where $l=\frac{n-2}{2}$, so $(a_{\sigma,\tau})^{n-1}$ is equivalent to $a_{\sigma,\tau}$ in $\HH^2(\Gal(\overline{\mathds{Q}}/\mathds{Q}),\overline{\mathds{Q}}^*)$. We have thus determined $\prescript{}{\xi_{k}}{\SL_n(k)}$.

It remains to compute $\prescript{}{\zeta}{\SL_n(\mathds{Q})}$ in the case where $\zeta$ is not inner. Note that since $\prescript{}{\zeta}{\SL_n(\mathds{R})}\simeq\SL_n(\mathds{R})$ we must have $d>0$. For all $\sigma\in\Gal(\overline{\mathds{Q}}/\mathds{Q}(\sqrt{d}))$ we have $\zeta(\sigma)\circ\sigma(\M)=\M$. This means that $\M\in\prescript{}{\zeta_{\mathds{Q}(\sqrt{d})}}{\SL_n(\mathds{Q}(\sqrt{d}))}$. Using Lemma \ref{lemma1}, there exist $a,b\in\mathds{N}$ and $\PP\in\SL(2,\overline{\mathds{Q}})$ such that $\xi(\sigma)=\Int(\PP^{-1}\T_{\sigma}^{a,b}\sigma(\PP))$ for every $\sigma\in\Gal(\overline{\mathds{Q}}/\mathds{Q})$. Up to conjugation we can assume that $\PP=\I_2$.

For any $\sigma\in\Gal(\overline{\mathds{Q}}/\mathds{Q})$ such that $\sigma(\sqrt{d})=-\sqrt{d}$ the equations becomes \begin{equation}
\label{equationhermitianform}
    \sigma(\M)^\top\J_n\tau_n(\T_{\sigma}^{a,b})^{-1}\M=\J_n\tau_n(\T_{\sigma}^{a,b})^{-1}
\end{equation} It remains to understand the Hermitian form $\J_n\tau_n(\T_{\sigma}^{a,b})^{-1}$.

Let $n$ be odd. By Proposition \ref{propHermitianQd}, $\sigma$-Hermitian forms over $\mathds{Q}(\sqrt{d})$ are classified by their rank and their discriminant. Direct computations show that $\J_n\tau_n(\T_{\sigma}^{a,b})^{-1}$ has rank $n$ and its discriminant is a square or minus a square. Up to multiplying the matrix by $-1$, which does not affect the unitary group, we see that it is congruent to $\I_n$ and the corresponding unitary groups are conjugate. 

Let $n$ be even. An element of $\prescript{}{\zeta_{\mathds{Q}}(\sqrt{d})}{\SL_n(\mathds{Q}(\sqrt{d}))}$ satisfies equation (\ref{equationhermitianform}) for one automorphism $\sigma\in\Gal(\overline{\mathds{Q}}/\mathds{Q})$ such that $\sigma(\sqrt{d})=-\sqrt{d}$ if and only if it satisfies it for all such $\sigma$. Up to change of basis, we can suppose that neither $a$ nor $b$ is a square (note that $\M_2(\mathds{Q})\simeq(5,5)_{\mathds{Q}}$). Under this assumption, $\sigma$ induces on $\prescript{}{\zeta_{\mathds{Q}(\sqrt{d})}}{\M_2(\mathds{Q}(\sqrt{d}))}$ the anti-involution $\overline{\phantom{s}}\otimes\tau$ where $\tau\in\Gal(\mathds{Q}(\sqrt{d})/\mathds{Q})$ is non-trivial and $\overline{\phantom{s}}$ is the conjugation of $\prescript{}{\zeta_{\mathds{Q}(\sqrt{d})}}{\M_2(\mathds{Q}(\sqrt{d}))}$. From \S 7 in \cite{Lewis_IsometryClassificationHermitianForms}, $\overline{\phantom{s}}\otimes\tau$-Hermitian forms of a quaternion algebra over $\mathds{Q}(\sqrt{d})$ are classified by their discriminant and their rank. Again, direct computations show that $\J_n\tau_n(\T_{\sigma}^{a,b})^{-1}$ has rank $n$, and its discriminant is a square up to sign. Thus it is congruent to \[\I_{\frac{n}{2}}\in\M(\text{\tiny{\( \frac{n}{2} \)}}, \prescript{}{\zeta_{\mathds{Q}(\sqrt{d})}}{\M_2(\mathds{Q}(\sqrt{d}))}\] and the corresponding unitary groups are conjugate.
\end{proof}

\subsection{Proof of Propositions \ref{proposition2} and \ref{proposition2.1}}

\begin{definition}
Let $\G$ be a linear $\mathds{Q}$-algebraic group and consider a continuous $1$-cocycle $\zeta:\Gal(\overline{\mathds{Q}}/\mathds{Q})\to\G(\overline{\mathds{Q}})$. We denote by $\prescript{}{\zeta}{\G}(\mathds{Z})$ any subgroup of $\prescript{}{\zeta}{\G}(\mathds{Q})$ commensurable with the $\mathds{Z}$-points of $\prescript{}{\zeta}{\G}(\mathds{Q})$.
\end{definition}

\begin{proof}[Proof of Propositions \ref{proposition2} and \ref{proposition2.1}]
If $\phi:\G\rightarrow\G'$ is a $\mathds{Q}$-morphism of algebraic groups over $\mathds{Q}$ then for any $\mathds{Q}$-arithmetic subgroup $\Gamma$ of $\G(\mathds{Q})$, $\phi(\Gamma)$ is contained in a $\mathds{Q}$-arithmetic subgroup of $\G'(\mathds{Q})$ (see Milne Proposition 5.2 of Appendix A \cite{Milne_LieAlgebrasAlgebraicGroupsLieGroups}). It implies that if $\xi$ and $\zeta$ are $1$-cocycles which are $\tau_n$-compatible, any $\mathds{Q}$-arithmetic subgroup of $\prescript{}{\xi}{\SL_2(\mathds{Q})}$ has image under $\tau_n$ which lies in a $\mathds{Q}$-arithmetic subgroup of $\prescript{}{\zeta}{\SL_n(\mathds{Q})}$ with $\zeta$ a 1-cocycle $\tau_n$-compatible with $\xi$. The $\mathds{Q}$-arithmetic subgroups of $\prescript{}{\zeta}{\SL_n(\mathds{Q})}$ are commensurable with groups isomorphic to
\begin{itemize}
    \setlength\itemsep{0cm}
    \item $\SL(n,\mathds{Z})$ if $n$ is odd and $\zeta$ is inner
    \item $\SU(\I_n,\sigma;\mathds{Z}[\sqrt{d}])$ if $n$ is odd and $\zeta$ is not inner
    \item $\SL(\frac{n}{2},\mathcal{O})$ if $n$ is even and $\zeta$ is inner
    \item $\SU(\I_{\frac{n}{2}},\overline{\phantom{s}}\otimes\sigma;\mathcal{O}\otimes_{\mathds{Z}}\mathds{Z}[\sqrt{d}])$ if $n$ is even, $\zeta$ is not inner and $\prescript{}{\xi}{\M_2(\mathds{Q})}\otimes_{\mathds{Q}}\mathds{Q}(\sqrt{d})\not\simeq\M_2(\mathds{Q}(\sqrt{d}))$
    \item $\SU(\I_n,\sigma,\mathds{Z}[\sqrt{d}])$ if $n$ is even, $\zeta$ is not inner and $\prescript{}{\xi}{\M_2(\mathds{Q})}\otimes_{\mathds{Q}}\mathds{Q}(\sqrt{d})\simeq\M_2(\mathds{Q}(\sqrt{d}))$
\end{itemize}
where $\mathcal{O}$ is an order of the quaternion algebra $\prescript{}{\xi}{\M_2(\mathds{Q})}$ and $d\in\mathds{N}$ square free. We get arithmetic subgroups of $\SL(n,\mathds{R})$ exactly when $\zeta$ is inner or when $\zeta$ is not inner and $d>0$.

Conversely let $\Gamma$ be a $\mathds{Q}$-arithmetic subgroup of $\SL(2,\mathds{R})$. Since every $\mathds{R}/\mathds{Q}$-form of $\SL_2$ is a $\overline{\mathds{Q}}/\mathds{Q}$-form of $\SL_2$ (see Proposition 2 of III.\S1 in \cite{Serre_GaloisCohomology}) there exists $$\xi:\Gal(\overline{\mathds{Q}}/\mathds{Q})\to\Aut(\SL_2(\overline{\mathds{Q}}))$$ a $1$-cocycle such that $\prescript{}{\xi}{\SL_2}(\mathds{Z})$ is commensurable with $\Gamma$. Let $\Lambda$ be a $\mathds{Q}$-arithmetic subgroup of $\SL(n,\mathds{R})$ and $$\zeta:\Gal(\overline{\mathds{Q}}/\mathds{Q})\to\Aut(\SL_n(\overline{\mathds{Q}}))$$ be a $1$-cocycle such that $\prescript{}{\zeta}{\SL_n}(\mathds{Z})$ is commensurable with $\Lambda$. We assume that $\tau_n(\Gamma)<\Lambda$ and we want to show that $\tau_n(\prescript{}{\xi}{\SL_2}(\mathds{Q}))<\prescript{}{\zeta}{\SL_n}(\mathds{Q})$. For any $\sigma\in\Gal(\overline{\mathds{Q}}/\mathds{Q})$, let
$$\tau^{\sigma}_n:\SL_2(\overline{\mathds{Q}})\to\SL_n(\overline{\mathds{Q}}),g\mapsto\zeta(\sigma)\circ\sigma\circ\tau_n\circ(\xi(\sigma)\circ\sigma)^{-1}(g).$$ This is an algebraic morphism that coincides with $\tau_n$ on a finite index subgroup of $\prescript{}{\xi}{\SL_2(\mathds{Z})}$. For any $\sigma$ consider
$$(\tau_n,\tau^{\sigma}_n):\SL_2(\overline{\mathds{Q}})\to\SL_n(\overline{\mathds{Q}})\times\SL_n(\overline{\mathds{Q}}).$$ Since $\{(g,h)\in\SL_n(\overline{\mathds{Q}})\times\SL_n(\overline{\mathds{Q}})|h=g\}$ is closed in the product for the Zariski topology, its inverse image by $(\tau_n,\tau^{\sigma}_n)$ has to be closed in $\SL_2(\overline{\mathds{Q}})$ so contains the Zariski-closure of $\prescript{}{\xi}{\SL_2(\mathds{Z})}$. Since any finite index subgroup of $\prescript{}{\xi}{\SL_2(\mathds{Z})}$ is Zariski-dense in $\SL_2(\overline{\mathds{Q}})$, we conclude that $\tau_n$ and $\tau^{\sigma}_n$ coincide on $\SL_2(\overline{\mathds{Q}})$.

For every $\sigma\in\Gal(\overline{\mathds{Q}}/\mathds{Q})$ and $g\in\prescript{}{\xi}{\SL_2}(\mathds{Q})$ we have $\tau_n(g)=\tau_n^{\sigma}(g)$. It means that $\tau_n(\prescript{}{\xi}{\SL_2}(\mathds{Q}))<\prescript{}{\zeta}{\SL_n}(\mathds{Q})$. Thus $\xi$ and $\zeta$ are $\tau_n$-compatible and Proposition \ref{proposition23} concludes the proof.
\end{proof}

\section{$\R/\QQ$-forms of other split real Lie groups}

In this section we establish analogues of Propositions \ref{proposition2} and \ref{proposition2.1} in the cases of $\Sp(2n,\mathds{R})$, $\SO(\J_n,\mathds{R})$ and $\G_2$.

\subsection{$\R/\QQ$-forms of $\Sp(2n,\R)$}

Let $n\geq2$.

\begin{lemma}
\label{Qformssymplectic}
Let $\xi:\Gal(\overline{\mathds{Q}}/\mathds{Q})\rightarrow\Aut(\SL_2(\overline{\mathds{Q}}))$ be a continuous $1$-cocycle such that $\prescript{}{\xi}{\SL_2(\mathds{R})}\simeq\SL(2,\mathds{R})$. Let $\zeta:\Gal(\overline{\mathds{Q}}/\mathds{Q})\rightarrow\Aut(\Sp_{2n}(\overline{\mathds{Q}}))$ be a continuous $1$-cocycle $\tau_{2n}$-compatible with $\xi$. Then
\begin{equation*}
    \prescript{}{\zeta}{\Sp_{2n}(\mathds{Q})}\simeq\SU(\I_{2n},\overline{\phantom{s}};\prescript{}{\xi}{\M_2(\mathds{Q})}).
\end{equation*}
\end{lemma}

\begin{proof}
Using the embedding $$\Aut(\Sp_{2n}(\overline{\mathds{Q}}))\simeq\PSp(2n,\overline{\mathds{Q}})\hookrightarrow\Aut(\SL_{2n}(\overline{\mathds{Q}})),$$ we see that $\prescript{}{\zeta}{\Sp_{2n}(\mathds{Q})}<\prescript{}{\zeta}{\SL_{2n}(\mathds{Q})}$.  All $\overline{\mathds{Q}}/\mathds{Q}$-forms of $\Sp_{2n}$ are isomorphic to $\Sp_{2n}$ or to $\HH(K)=\SU(\I_n,\overline{\phantom{s}},A\otimes_{\mathds{Q}}K)$ for any field extension $K$ of $\mathds{Q}$ and with a quaternion division algebra $A$ over $\mathds{Q}$ (since there is only one non-degenerate $\overline{\phantom{s}}$-Hermitian form on $A$, see Proposition \ref{propHermitianconj}). There can only be one $1$-cocycle $\tau_{2n}$-compatible with $\xi$ since all automorphisms of $\Sp_{2n}(\overline{\mathds{Q}})$ are inner (see Proposition \ref{proposition1}). Using Proposition \ref{proposition23} we deduce that $A\simeq\prescript{}{\xi}{\M_2(\mathds{Q})}$.
\end{proof}

\begin{proposition}
Let $\Gamma$ be an arithmetic subgroup of $\SL(2,\mathds{R})$. Suppose $\Gamma$ is commensurable with $\mathcal{O}^1$ where $\mathcal{O}$ is an order of a quaternion division algebra over $\mathds{Q}$.

Then $\tau_{2n}(\Gamma)$ lies in a subgroup of $\Sp(2n,\mathds{R})$ commensurable with a conjugate of $\SU(\I_{n},\overline{\phantom{s}};\mathcal{O})$.

Furthermore this is the only $\mathds{Q}$-arithmetic subgroup of $\Sp(2n,\mathds{R})$ that contains $\tau_n(\Gamma)$ up to commensurability.
\end{proposition}

\begin{proof}
Thanks to Proposition 5.2 of Appendix A in Milne \cite{Milne_LieAlgebrasAlgebraicGroupsLieGroups} $\mathds{Q}$-arithmetic subgroups of $\prescript{}{\xi}{\SL_2(\mathds{Q})}$ have image under $\tau_{2n}$ which lies in a $\mathds{Q}$-arithmetic subgroup of $\prescript{}{\zeta}{\Sp_{2n}(\overline{\mathds{Q}})}$ with $\zeta$ a 1-cocycle $\tau_{2n}$-compatible with $\xi$. Lemma \ref{Qformssymplectic} determines in which arithmetic subgroups $\tau_n(\Gamma)$ lie in. They are all arithmetic subgroups of $\Sp(\J_{2n},\mathds{R})$.

The converse is proven the same way as in the proof of Propositions \ref{proposition2} and \ref{proposition2.1}.
\end{proof}

\subsection{$\mathds{R}/\mathds{Q}$-forms of $\SO(k+1,k)$}

Let $n=2k+1\geq3$ be odd. By Theorem 2.8 of \cite{Platonov_AlgebraicgroupsNumbertheory}, $\Aut(\SO(\J_n,\overline{\mathds{Q}}))\simeq\SO(\J_n,\overline{\mathds{Q}})$. Hence isomorphism classes of $\overline{\mathds{Q}}/\mathds{Q}$-forms of $\SO(\J_n)$ are classified by $\HH^1(\Gal(\overline{\mathds{Q}}/\mathds{Q}),\SO(\J_n,\overline{\mathds{Q}}))$, as explained in Section \ref{sectionQformsSL}. Proposition 2.8 in \cite{Platonov_AlgebraicgroupsNumbertheory} states that elements of $\HH^1(\Gal(\overline{\mathds{Q}}/\mathds{Q}),\SO(\J_n,\overline{\mathds{Q}}))$ are in one-to-one correspondence with $\mathds{Q}$-equivalence classes of quadratic forms over $\mathds{Q}^n$ that have discriminant $1$.

The correspondence works as follows. If $\B\in\SL(n,\mathds{Q})$ is a symmetric matrix, then the associated equivalence class of $1$-cocycle is the one defined by the $\overline{\mathds{Q}}/\mathds{Q}$-form $\SO(\B)$. Conversely let $\xi$ is a $1$-cocycle. The embedding $\SO(\J_n,\overline{\mathds{Q}})\hookrightarrow\GL(n,\overline{\mathds{Q}})$ induces a map
\begin{equation*}
    \HH^1(\Gal(\overline{\mathds{Q}}/\mathds{Q}),\SO(\J_n,\overline{\mathds{Q}}))\xrightarrow{\phi}\HH^1(\Gal(\overline{\mathds{Q}}/\mathds{Q}),\GL(n,\overline{\mathds{Q}})).
\end{equation*}
The codomain of $\phi$ is trivial by Hilbert's Theorem 90 (see Lemma 2.2 in \cite{Platonov_AlgebraicgroupsNumbertheory}). Thus there exists $\SSS\in\GL(n,\overline{\mathds{Q}})$ such that for all $\sigma\in\Gal(\overline{\mathds{Q}}/\mathds{Q})$ we have $\phi(\xi)(\sigma)=\SSS^{-1}\sigma(\SSS)$. The symmetric matrix associated to $\xi$ is $\SSS^{-\top}\J_n\SSS^{-1}$ (note that we are using the convention $\SO(\J_n,\overline{\mathds{Q}})=\{\M\in\SL(n,\overline{\mathds{Q}})\ |\ \M^\top\J_n\M=\J_n\}$).

Here is the way to determine the matrix $\SSS$ which comes from an examination of the proof of Hilbert's Theorem 90. Recall that $\xi$ belongs to the direct limit $\varinjlim\HH^1(\Gal(K/\mathds{Q}),\SO(\J_n,K))$, where $K$ runs through all finite Galois extensions of $\mathds{Q}$. Hence there exists such a $K$ with $\xi\in\HH^1(\Gal(K/\mathds{Q}),\SO(\J_n,K))$. Define
\begin{equation*}
    f:K^n\to K^n,\ x\mapsto\sum_{\sigma\in\Gal(K/\mathds{Q})}\xi(\sigma)\sigma(x).
\end{equation*}
Pick $v_1,...,v_n\in K^n$ such that $(f(v_1),...,f(v_n))$ is a basis of $K^n$. Then $\SSS=(f(v_1)|...|f(v_n))^{-1}$.

Quadratic forms over $\mathds{Q}$ are classified up to equivalence by global and local invariants. The main reference for this subject is Serre's book \cite{Serre_CourseArithmetic}. If $q(x)=\sum a_ix_i^2$ is a non-degenerate quadratic form  over $\mathds{Q}$ and $p$ is a prime number, define its \emph{Hasse invariant in $\mathds{Q}_p$} as
\begin{equation*}
    \mathcal{E}_p(q)=\bigotimes_{i<j}(a_i,a_j)_{\mathds{Q}_p}\in\Br(\mathds{Q}_p).
\end{equation*}
The Hasse-Minkowski Theorem (Theorem 9 in \cite{Serre_CourseArithmetic}) now states that two quadratic forms over $\mathds{Q}$ are equivalent if and only if they have the same discriminant, the same signature and the same Hasse invariant in $\mathds{Q}_p$ for all primes $p$.

We can compute the invariants of $\J_n$. Its discriminant is a square. Its signature is $(k+1,k)$ if $n\equiv1[4]$ and $(k,k+1)$ if $n\equiv3[4]$. Using the fact that $$\mathcal{E}(q\perp q')=\mathcal{E}_p(q)\otimes\mathcal{E}_p(q')\otimes(\disc(q),\disc(q'))_{\mathds{Q}_p}$$ for all primes $p$, we can show that $\mathcal{E}_p(\J_n)=1$ for $p\neq2$ and that $\mathcal{E}_2(\J_n)=1$ if $n\equiv\pm1[8]$ and $\mathcal{E}_2(\J_n)=(-1,-1)_{\mathds{Q}_2}$ if $n\equiv\pm3[8]$.

\begin{lemma}
\label{Qformsorthogonal}
Let $\xi:\Gal(\overline{\mathds{Q}}/\mathds{Q})\to\Aut(\SL_2(\overline{\mathds{Q}}))$ be a continuous $1$-cocycle such that $\xi(\sigma)=\Int(\PP^{-1}\T_{\sigma}^{a,b}\sigma(\PP))$ for all $\sigma$ with $a,b\in\mathds{N}$. Let $n=2k+1\geq3$ be odd and $\zeta:\Gal(\overline{\mathds{Q}}/\mathds{Q})\to\Aut(\SO(\J_n,\overline{\mathds{Q}}))$ be a continuous $1$-cocycle $\tau_n$-compatible with $\xi$. Then
\begin{equation*}
    \prescript{}{\zeta}{\SO(\J_n)(\mathds{Q})}\simeq\left\{\begin{array}{ll}
        \SO(\J_n,\mathds{Q}) & \mbox{if $n\equiv\pm1[8]$} \\
        \SO(\Q,\mathds{Q}) & \mbox{if $n\equiv\pm3[8]$}
    \end{array}\right.
\end{equation*}
where $\Q\in\SL(n,\mathds{Q})$ is a symmetric matrix of signature equal to the signature of $\J_n$ and of Hasse invariant $\mathcal{E}_p(\Q)=(a,b)_{\mathds{Q}_p}\otimes(-1,-1)_{\mathds{Q}_p}$ for all primes $p$.
\end{lemma}

\begin{proof}
Using $$\Aut(\SO(\J_n,\overline{\mathds{Q}}))\simeq\SO(\J_n,\overline{\mathds{Q}})\xhookrightarrow{i}\PSL(n,\overline{\mathds{Q}})$$ we consider $i\circ\zeta$ as $1$-cocycle with value in $\Aut(\SL_n(\overline{\mathds{Q}}))$ which is $\tau_n$-compatible with $\xi$. Proposition \ref{proposition1} implies that for all $\sigma$ $$i(\zeta(\sigma))=\Int(\tau_n(\PP^{-1}\T_{\sigma}^{a,b}\sigma(\PP))).$$ Up to conjugation, we can assume that $\PP=\I_2$.

Let $K=\mathds{Q}(\sqrt{a},\sqrt{b})$. Then $\zeta\in\HH^1(\Gal(K/\mathds{Q}),\SO(\J_n,K))$. We define $$f:K^n\to K^n,\ x\mapsto\sum_{\sigma\in\Gal(K/\mathds{Q})}\T_{\sigma}^{a,b}\sigma(x).$$ If $K$ is a degree $2$ extension of $\mathds{Q}$ let $\sigma\in\Gal(K/\mathds{Q})$ be the non-trivial element. If $K$ is a degree $4$ extension of $\mathds{Q}$, let $\sigma\in\Gal(K/\mathds{Q})$ be such that $\sigma(\sqrt{a})=-\sqrt{a}$ and $\sigma(\sqrt{b})=\sqrt{b}$ and let $\tau\in\Gal(K/\mathds{Q})$ be such that $\tau(\sqrt{a})=\sqrt{a}$ and $\tau(\sqrt{b})=-\sqrt{b}$. We note $(e_1,...,e_n)$ the canonical basis of $K^n$.

Assume that $n\equiv1[4]$. If $K=\mathds{Q}$ then the symmetric matrix associated to $\zeta$ is $\J_n$. We suppose now that $K$ is a degree $2$ extension of $\mathds{Q}$. If $a$ or $b$ is a square, then $(a,b)_{\mathds{Q}}=1$ and up to change of basis we can assume that both of them are squares. This has been treated in the previous case. Suppose that $a$ and $b$ are not squares. We have
$$f:K^n\to K^n,\ x\mapsto x+\begin{psmallmatrix}
&&&1\\
&&-1&\\
&...&&\\
1&&&
\end{psmallmatrix}\sigma(x).$$ We then define for all $1\leq i\leq\frac{k}{2}$
$$v_{2i-1}=\frac{e_{2i-1}+e_{2k-2i+3}}{2},\ v_{2i}=\frac{e_{2i}-e_{2k-2i+2}}{2},$$ $$v_{2k-2i+2}=\frac{\sqrt{a}e_{2i}+\sqrt{a}e_{2k-2i+2}}{2},\ v_{2k-2i+3}=\frac{\sqrt{a}e_{2i-1}-\sqrt{a}e_{2k-2i+3}}{2}$$ and $v_{k+1}=\frac{e_{k+1}}{2}$.
Let
$$\SSS^{-1}=(f(v_1),...,f(v_n))=\begin{pmatrix}
1&&&&&&\sqrt{a}\\
&1&&&&\sqrt{a}&\\
&&...&&...&&\\
&&&1&&&\\
&&...&&...&&\\
&-1&&&&\sqrt{a}&\\
1&&&&&&-\sqrt{a}
\end{pmatrix}.$$
The symmetric matrix associated to $\zeta$ is $\SSS^{-\top}\J_n\SSS^{-1}=$
$$\begin{pmatrix}2(n-1)!&&&&&&\\
&2(n-2)!&&&&\\
&&...&&&&\\
&&&k!k!&&&\\
&&&&...&&\\
&&&&&-2a(n-2)!&\\
&&&&&&-2a(n-1)!
\end{pmatrix}$$
which has Hasse invariant for any prime $p$
$$(-1,-1)_{\mathds{Q}_p}^{\frac{n-1}{4}}\otimes\bigotimes_{j=1,3,...,k-1}(a,-j(n-j))_{\mathds{Q}_p}\simeq(-1,-1)_{\mathds{Q}_p}^{\frac{n-1}{4}}\otimes(a,(-1)^{\frac{n-1}{4}})_{\mathds{Q}_p}$$
since $\prod_{j=1,3,...,k-1}j(n-j)$ is a square as can be shown by induction. Thus if $n\equiv1[8]$ then $\SSS^{-\top}\J_n\SSS^{-1}$ is equivalent to $\J_n$. If $n\equiv-3[8]$ then $\SSS^{-\top}\J_n\SSS^{-1}$ has Hasse invariant $(-1,-1)_{\mathds{Q}_p}\otimes(a,a)_{\mathds{Q}_p}$ for all primes $p$.

We suppose that $K$ is a degree $4$ extension of $\mathds{Q}$. Then $f:K^n\to K^n$
$$x\mapsto x+\begin{psmallmatrix}
1&&&\\
&-1&&\\
&&...&\\
&&&1
\end{psmallmatrix}\tau(x)+\begin{psmallmatrix}
&&&1\\
&&1&\\
&...&&\\
1&&&
\end{psmallmatrix}\sigma(x)+\begin{psmallmatrix}
&&&1\\
&&-1&\\
&...&&\\
1&&&\\
\end{psmallmatrix}\sigma(\tau(x)).$$
For all $1\leq i\leq \frac{k}{2}$ we let
$$v_{2i-1}=\frac{\sqrt{a}e_{2i-1}}{2},\ v_{2i}=\frac{\sqrt{b}e_{2i}}{2},$$
$$v_{2k-2i+2}=-\frac{\sqrt{ab}}{2}e_{2k-2i+2},\ v_{2k-2i+3}=\frac{e_{2k-2i+3}}{2}$$
and $v_{k+1}=\frac{e_{k+1}}{4}$. Let
$$\SSS^{-1}=(f(v_1),...,f(v_n))=\begin{pmatrix}
\sqrt{a}&&&&&&1\\
&\sqrt{b}&&&&\sqrt{ab}&\\
&&...&&...&&\\
&&&1&&&\\
&&...&&...&&\\
&\sqrt{b}&&&&-\sqrt{ab}&\\
-\sqrt{a}&&&&&&1
\end{pmatrix}.$$ The symmetric matrix associated to $\zeta$ is $\SSS^{-\top}\J_n\SSS^{-1}=$
$$\begin{pmatrix}
-2a(n-1)!&&&&&&\\
&-2b(n-2)!&&&&&\\
&&...&&&&\\
&&&k!k!&&&\\
&&&&...&&\\
&&&&&2ab(n-2)!&\\
&&&&&&2(n-1)!
\end{pmatrix}$$
which has Hasse invariant
$$(-1,-1)^{\frac{n-1}{4}}_{\mathds{Q}_p}\otimes\bigotimes_{j=1,3,...,k-1}(a,bj(n-j))_{\mathds{Q}_p}\simeq(-1,-1)^{\frac{n-1}{4}}_{\mathds{Q}_p}\otimes(a,b^{\frac{n-1}{4}})_{\mathds{Q}_p}.$$ Thus if $n\equiv1[8]$, $\SSS^{-\top}\J_n\SSS^{-1}$ is equivalent to $\J_n$. If $n\equiv-3[8]$, $\SSS^{-\top}\J_n\SSS^{-1}$ has Hasse invariant $(-1,-1)_{\mathds{Q}_p}\otimes(a,b)_{\mathds{Q}_p}$ for all primes $p$.

Assume that $n\equiv3[4]$. If $K=\mathds{Q}$, the symmetric matrix associated to $\zeta$ is $\J_n$.
We now suppose that $K$ is a degree 2 extension of $\mathds{Q}$. If $a$ or $b$ is a square then $(a,b)_{\mathds{Q}}=1$ and up to change of basis we can assume that both of them are squares. This has been treated in the previous case. Assume that none of them is a square. We have
$$f:K^n\to K^n,\ x\mapsto x+\begin{psmallmatrix}
&&&1\\
&&-1&\\
&...&&\\
1&&&\\
\end{psmallmatrix}\sigma(x).$$ We then define for all $1\leq i\leq \frac{k+1}{2}$
$$v_{2i-1}=e_{2i-1},\ v_{2i}=e_{2i},$$
$$v_{2k-2i+2}=\sqrt{a}e_{2k-2i+2},\ v_{2k-2i+3}=-\sqrt{a}e_{2k-2i+3}$$ and $v_{k+1}=\frac{\sqrt{a}e_{k+1}}{2}$. Let
$$\SSS^{-1}=(f(v_1),...,f(v_n))=\begin{pmatrix}
1&&&&&&\sqrt{a}\\
&1&&&&\sqrt{a}&\\
&&...&&...&&\\
&&&\sqrt{a}&&&\\
&&...&&...&&\\
&-1&&&&\sqrt{a}&\\
1&&&&&&-\sqrt{a}
\end{pmatrix}.$$
The symmetric matrix associated to $\zeta$ is $\SSS^{-\top}\J_n\SSS^{-1}=$
$$\begin{pmatrix}
2(n-1)!&&&&&&\\
&2(n-2)!&&&&&\\
&&...&&&&\\
&&&-ak!k!&&&\\
&&&&...&&\\
&&&&&-2a(n-2)!&\\
&&&&&&-2a(n-1)!
\end{pmatrix}$$
which has Hasse invariant
\begin{align*}&(-1,-1)_{\mathds{Q}_p}^{\frac{n+1}{4}}\otimes(a,-2a)_{\mathds{Q}_p}\otimes\bigotimes_{j=1,3,...,k}(a,-j(n-j))_{\mathds{Q}_p}\\
&\simeq(-1,-1)^{\frac{n+1}{4}}_{\mathds{Q}_p}\otimes(a,(-1)^{\frac{n-3}{4}}a)_{\mathds{Q}_p}\end{align*} for all primes $p$ since $2\prod_{j=1,3,...,k}j(n-j)$ is a square as can be shown by induction. Thus if $n\equiv-1[8]$ then $\SSS^{-\top}\J_n\SSS^{-1}$ is equivalent to $\J_n$. If $n\equiv3[8]$ then $\SSS^{-\top}\J_n\SSS^{-1}$ has Hasse invariant $(-1,-1)_{\mathds{Q}_p}\otimes(a,a)_{\mathds{Q}_p}.$

We suppose that $K$ is a degree $4$ extension of $\mathds{Q}$. Then $f:K^n\to K^n$ $$x\mapsto x+\begin{psmallmatrix}
-1&&&\\
&1&&\\
&&...&\\
&&&-1
\end{psmallmatrix}\tau(x)+\begin{psmallmatrix}
&&&-1\\
&&-1&\\
&...&&\\
-1&&&
\end{psmallmatrix}\sigma(x)+\begin{psmallmatrix}
&&&1\\
&&-1&\\
&...&&\\
1&&&
\end{psmallmatrix}\sigma(\tau(x)).$$
We then define for all $1\leq i\leq \frac{k+1}{2}$
$$v_{2i-1}=\frac{\sqrt{b}e_{2i-1}}{2},\ v_{2i}=\frac{e_{2i}}{2},$$
$$v_{2k-2i+2}=\frac{\sqrt{a}e_{2k-2i+2}}{2},\ v_{2k-2i+3}=\frac{\sqrt{ab}e_{2k-2i+3}}{2}$$ and $v_{k+1}=\frac{\sqrt{a}e_{k+1}}{4}.$
Let
$$\SSS^{-1}=(f(v_1),...,f(v_n))=\begin{pmatrix}
\sqrt{b}&&&&&&\sqrt{ab}\\
&1&&&&\sqrt{a}&\\
&&...&&...&&\\
&&&\sqrt{a}&&&\\
&&...&&...&&\\
&-1&&&&\sqrt{a}&\\
-\sqrt{b}&&&&&&\sqrt{ab}
\end{pmatrix}.$$
The symmetric matrix associated to $\zeta$ is $\SSS^{-\top}\J_n\SSS^{-1}=$
$$\begin{pmatrix}
-2b(n-1)!&&&&&&\\
&2(n-2)!&&&&&\\
&&...&&&&\\
&&&-ak!k!&&&\\
&&&&...&&\\
&&&&&-2a(n-2)!&\\
&&&&&&2ab(n-1)!
\end{pmatrix}$$
which has Hasse invariant
$$(-1,-1)_{\mathds{Q}_p}^{\frac{n+1}{4}}\otimes(a,2)_{\mathds{Q}_p}\otimes\bigotimes_{j=1,3,...,k}(a,bj(n-j))_{\mathds{Q}_p}\simeq(-1,-1)_{\mathds{Q}_p}^{\frac{n+1}{4}}\otimes(a,b^{\frac{n+1}{4}})_{\mathds{Q}_p}$$ for all primes $p$. Thus if $n\equiv-1[8]$, $\SSS^{-\top}\J_n\SSS^{-1}$ is equivalent to $\J_n$. If $n\equiv3[8]$, $\SSS^{-\top}\J_n\SSS^{-1}$ has for Hasse invariant $(-1,-1)_{\mathds{Q}_p}\otimes(a,b)_{\mathds{Q}_p}$.
\end{proof}

\begin{proposition}
Let $\Gamma$ be a $\mathds{Q}$-arithmetic subgroup of $\SL(2,\mathds{R})$ and $n=2k+1\geq3$. Suppose that $\Gamma$ is commensurable with the norm $1$ elements of an order of a quaternion algebra $(a,b)_{\mathds{Q}}$.

If $n\equiv\pm1[8]$, then $\tau_n(\Gamma)$ lies in a subgroup of $\SO(\J_n,\mathds{R})$ commensurable with a conjugate of $\SO(\J_n,\mathds{Z})$.

If $n\equiv\pm3[8]$, then $\tau_n(\Gamma)$ lies in a subgroup of $\SO(\J_n,\mathds{R})$ commensurable with a conjugate of $\SO(\Q,\mathds{Z})$ for $\Q\in\SL(n,\mathds{Q})$ a symmetric matrix of signature equal to the signature of $\J_n$ and of Hasse invariant $\mathcal{E}_p(\Q)=(a,b)_{\mathds{Q}_p}\otimes(-1,-1)_{\mathds{Q}_p}$ for every prime $p$.

Furthermore these are the only $\mathds{Q}$-arithmetic subgroups of $\SO(\J_n,\mathds{R})$ that contain $\tau_n(\Gamma)$ up to commensurability.
\end{proposition}

\begin{proof}
Thanks to Proposition 5.2. in Milne \cite{Milne_LieAlgebrasAlgebraicGroupsLieGroups} $\mathds{Q}$-arithmetic subgroups of $\prescript{}{\xi}{\SL_2(\mathds{Q})}$ have image under $\tau_n$ which lies in a $\mathds{Q}$-arithmetic subgroup of $\prescript{}{\zeta}{\SO(\J_n,\overline{\mathds{Q}})}$ with $\zeta$ a 1-cocycle $\tau_n$-compatible with $\xi$. Lemma \ref{Qformsorthogonal} determines in which $\mathds{Q}$-arithmetic subgroups $\tau_n(\Gamma)$ lie in. They are all arithmetic subgroups of $\SO(\J_n,\mathds{R})$.

The converse is proven the same way as in the proof of Proposition \ref{proposition2}.
\end{proof}

\subsection{$\mathds{R}/\mathds{Q}$-forms of $\G_2$}

We now give the corresponding result for $\G_2$.

\begin{definition}
An \emph{octonion algebra} over a field $F$ of characteristic not equal to $2$ is a unital but non associative algebra of dimension $8$ over $F$ such that there exists a nondegenerate quadratic form $N$ on the algebra satisfying $N(xy)=N(x)N(y)$ for all $x,y$. We say that the octonion algebra is \emph{split} if $N$ is isotropic.
\end{definition}

\begin{definition}
\label{definitionG2}
Let $R$ be a ring. Denote by $\times:R^7\times R^7\rightarrow R^7$ the map
\begin{equation*}
    \begin{pmatrix}
    x_1\\
    x_2\\
    x_3\\
    x_4\\
    x_5\\
    x_6\\
    x_7
    \end{pmatrix},
    \begin{pmatrix}
    y_1\\
    y_2\\
    y_3\\
    y_4\\
    y_5\\
    y_6\\
    y_7
    \end{pmatrix}\mapsto
    \begin{pmatrix}
    6(x_1y_4-x_4y_1)-4(x_2y_3-x_3y_2)\\
    24(x_1y_5-x_5y_1)-6(x_2y_4-x_4y_2)\\
    60(x_1y_6-x_6y_1)-6(x_3y_4-x_4y_3)\\
    120(x_1y_7-x_7y_1)+20(x_2y_6-x_6y_2)-8(x_3y_5-x_5y_3)\\
    60(x_2y_7-x_7y_2)-6(x_4y_5-x_5y_4)\\
    24(x_3y_7-x_7y_3)-6(x_4y_6-x_6y_4)\\
    6(x_4y_7-x_7y_4)-4(x_5y_6-x_6y_5)
    \end{pmatrix}
\end{equation*}
Define $\textbf{G}_2(R)=\{\M\in\SO(\J_7,R)|\M(x\times y)=\M x\times\M y,\ \forall x,y\in R^7\}$.
\end{definition}

It appears that $\textbf{G}_2(\mathds{R})$ is a Lie group isomorphic to the connected centerless real split form of type $\G_2$. To prove this, we should show that $\textbf{G}_2(\mathds{R})$ is the automorphism group of the unique split octonion algebra over $\mathds{R}$ (see Theorem 1.8.1 in Springer and Veldkamp's book \cite{Springer_OctonionJordanAlgebrasExceptionalGroups}). Let $\mathds{O}=\mathds{R}\oplus\mathds{R}^7$. We define the product on $\mathds{O}$ by
$$(t,v)(s,w)=(ts-v^\top\J_7w,tw+sv+v\times w)$$ for all $t,s\in\mathds{R}$ and $v,w\in\mathds{R}^7$. With this product, $\mathds{O}$ has a structure of a unital algebra over $\mathds{R}$. Denote by $N:\mathds{O}\rightarrow\mathds{R}$ the map
$$N(t,v)=t^2+v^\top\J_7 v.$$ This is a nondegenerate quadratic form on $\mathds{O}$.
We can check that  $$v^\top \J_7(v\times w)=0$$ and $$(v\times w)^\top\J_7(v\times w)=(v^\top\J_7v)(w^\top\J_7w)-(v^\top\J_7w)^2$$ for all $v,w\in\mathds{R}^7$. Those two properties together imply that $N(xy)=N(x)N(y)$. Since $N$ is isotropic, $\mathds{O}$ is isomorphic to the unique split octonion algebra over $\mathds{R}$.

Let us show that $\textbf{G}_2(\mathds{R})\simeq\Aut(\mathds{O})$ to conclude that $\textbf{G}_2(\mathds{R})$ is a Lie group of type $\G_2$.
Pick $g\in\textbf{G}_2(\mathds{R})$ and define its action on $\mathds{O}$ by $g.(t,v)=(t,g.v)$ for all $(t,v)\in\mathds{O}$. Then for all $x,y\in\mathds{O}$ $g(xy)=g(x)g(y)$. Thus $g$ defines an automorphism of $\mathds{O}$. Suppose $\phi\in\Aut(\mathds{O})$. Then $\phi(1,0)=(1,0)$ and thus $\phi$ fixes $\mathds{R}\oplus\{0\}$. Since $$N(t,v)=(t,v)(t,-v)$$ $N(\phi(x))=N(x)$ for all $x\in\mathds{O}$. This implies that $\phi$ preserves the orthogonal of $\mathds{R}\oplus\{0\}$ for $N$, i.e. $\phi$ preserves $\mathds{R}^7$. For any $v,w\in\mathds{R}^7$ we have
\begin{align*}
    &\phi(0,v)\phi(0,w)=\phi((0,v)(0,w))\\
    \implies&(0,\phi_{|\mathds{R}^7}(v))(0,\phi_{|\mathds{R}^7}(w))=\phi(-v^\top\J_7w,v\times w)\\
    \implies&(-\phi_{|\mathds{R}^7}(v)^\top\J_7\phi_{|\mathds{R}^7}(w),\phi_{|\mathds{R}^7}(v)\times\phi_{|\mathds{R}^7}(w))=(-v^\top\J_7w,\phi_{|\mathds{R}^7}(v\times w))
\end{align*}
which imply that $\phi_{|\mathds{R}^7}\in\textbf{G}_2(\mathds{R})$.

\begin{remark}
We can show that $\tau_7(\PSL(2,\mathds{R}))<\textbf{G}_2(\mathds{R})$ by checking it only for $\PSL(2,\mathds{Z})$, since the latter is Zariski-dense in $\PSL(2,\mathds{R})$.
\end{remark}

Denote by $\mathds{O}_{\overline{\mathds{Q}}}$ the unique octonion algebra over $\overline{\mathds{Q}}$ (see \S1.10 in \cite{Springer_OctonionJordanAlgebrasExceptionalGroups}). Since $\Aut(\textbf{G}_2(\overline{\mathds{Q}}))\simeq\textbf{G}_2(\overline{\mathds{Q}})$ (see Theorem 2.8 in \cite{Platonov_AlgebraicgroupsNumbertheory}) \begin{align*}\HH^1(\Gal(\overline{\mathds{Q}}/\mathds{Q}),\Aut(\textbf{G}_2(\overline{\mathds{Q}})))&\simeq\HH^1(\Gal(\overline{\mathds{Q}}/\mathds{Q}),\textbf{G}_2(\overline{\mathds{Q}})))\\
&\simeq\HH^1(\Gal(\overline{\mathds{Q}}/\mathds{Q}),\Aut(\mathds{O}_{\overline{\mathds{Q}}})).
\end{align*}
The latter has only two elements for there are only two octonions algebras over $\mathds{Q}$, as stated in \S1.10 of \cite{Springer_OctonionJordanAlgebrasExceptionalGroups}. Since there are two connected centerless simple real Lie groups of type $\G_2$, the split one and the compact one, this implies that $\textbf{G}_2$ has only one $\mathds{R}/\mathds{Q}$-form which is $\textbf{G}_2$ itself.

\begin{remark}
\label{remarkG2lattice}
Since $\textbf{G}_2$ has only one $\mathds{R}/\mathds{Q}$-form, all its $\mathds{Q}$-arithmetic subgroups are commensurable up to conjugation. In particular, all its non-uniform lattices are commensurable up to conjugation (see Corollary 5.3.2 in \cite{Morris_IntroductionArithmeticGroups}).
\end{remark}

\begin{proposition}
Let $\Gamma$ be a $\mathds{Q}$-arithmetic subgroup of $\SL(2,\mathds{R})$. Then $\tau_7(\Gamma)$ lies in a subgroup of \emph{$\textbf{G}_2(\mathds{R})$} commensurable with a conjugate of \emph{$\textbf{G}_2(\mathds{Z})$}.

Furthermore this is the only $\mathds{Q}$-arithmetic subgroups of \emph{$\textbf{G}_2(\mathds{R})$} that contains $\tau_7(\Gamma)$ up to commensurability.
\end{proposition}

\begin{proof}
Suppose that $\Gamma$ is the $\mathds{Z}$-points of a $\mathds{Q}$-form of $\SL_2(\mathds{R})$ associated to the $1$-cocycle $\xi$. Thanks to Proposition \ref{proposition1}, there is only one $1$-cocycle
$$\zeta:\Gal(\overline{\mathds{Q}}/\mathds{Q})\to\textbf{G}_2(\overline{\mathds{Q}})$$ which is $\tau_7$-compatible with $\xi$. We have $\prescript{}{\zeta}{\textbf{G}_2}(\mathds{Q})\simeq\textbf{G}_2(\mathds{Q}).$ Furthermore $\tau_7(\Gamma)<\textbf{G}_2(\mathds{Z})$.

The converse is proven the same way as in the proof of Proposition \ref{proposition2}.
\end{proof}

\section{Generalities on bending}
\label{sectionbending}

Let $\Gamma$ be a cocompact $\mathds{Q}$-arithmetic subgroup of $\SL(2,\mathds{R})$. Thanks to Propositions \ref{proposition2} and \ref{proposition2.1}, there exists a $\mathds{Q}$-arithmetic subgroup of $\SL(n,\mathds{R})$ $\Lambda$ such that $\tau_n(\Gamma)<\Lambda$. This implies that $\tau_n(\Gamma/\{\pm\I_2\})<\Lambda$ when $n$ is odd and $\tau_n(\Gamma/\{\pm\I_2\})<\Lambda/\{\pm\I_n\}$ when $n$ is even. There is a finite index subgroup of $\Gamma/\{\pm\I_2\}$ which is torsion free and hence a surface group. Thus $\tau_n$ induces a representation of a surface group into an arithmetic subgroup of $\PSL(n,\mathds{R})$ which, by definition, is a Hitchin representation. The image of this representation is not Zariski-dense in $\PSL(n,\mathds{R})$ since it lies in $\tau_n(\PGL(2,\mathds{R}))$. We will deform it so that it becomes Zariski-dense. The technique used is called \emph{bending}, as introduced by Johnson and Millson \cite{Johnson_DeformationSpacesCompactHyperbolicManifolds}.

Let $S$ be a closed orientable surface of genus at least 2 and $\gamma$ be a simple closed curve on $S$. We say that $\gamma$ is \emph{separating} if the complement of its image has two connected components, otherwise it is \emph{non-separating}.
Let $j:\pi_1(S)\rightarrow\PSL(2,\mathds{R})$ be a discrete and faithful representation. Let $\rho=\tau_n\circ j$. Choose $\B\in\PSL(n,\mathds{R})$ which commutes with $\rho([\gamma])$.

Suppose that $\gamma$ is separating and denote by $S_1$ and $S_2$ the two connected components of the complement of its image. Van Kampen's theorem states that the fundamental group of $S$ is an amalgamated product:
\begin{equation*}
    \pi_1(S)=\pi_1(S_1)*_{[\gamma]}\pi_1(S_2).
\end{equation*}
Define
\begin{align*}
    \rho^1:\pi_1(S_1)&\rightarrow\PSL(n,\mathds{R})\ g\mapsto\rho(g)\\
    \rho^2:\pi_1(S_2)&\rightarrow\PSL(n,\mathds{R})\
    g\mapsto\B\rho(g)\B^{-1}.
\end{align*}
Together they induce a new representation
\begin{equation*}
    \rho_{\B}:\pi_1(S)\rightarrow\PSL(n,\mathds{R}).
\end{equation*}
Suppose that $\gamma$ is non-separating and denote by $S_1$ the complement of its image. Then the fundamental group of $S$ is an HNN-extension of the fundamental group of $S_1$. More precisely let $T$ be a tubular neighborhood of the image of $\gamma$ in $S$. The curve $\gamma$ separates $T$ into two connected components which we denote by $T_1$ and $T_2$. Denote by $i:T_1\hookrightarrow S_1$ the inclusion. Pick $p_1\in T_1$, $p_2\in T_2$ and $\eta:[0;1]\to S$ a path from $p_1$ to $p_2$ that does not intersect $\gamma$. Let $$\phi:\pi_1(T_2,p_2)\xrightarrow{\sim}\pi_1(T_1,p_1),\ [\gamma]\mapsto[\eta^{-1}\circ\gamma\circ\eta].$$
Assume that $\pi_1(S_1,p_1)=\langle g_1,...,g_k \rangle$. Then
\begin{equation*}
    \pi_1(S)=\langle g_1,...,g_k,s\ |\ i_*(\phi(g))=s^{-1}i_*(g)s\ \forall g\in\pi_1(T_1,p_1)\rangle.
\end{equation*}
Define
\begin{align*}
    \rho^1:\ &\pi_1(S_1,p_1)\rightarrow\PSL(n,\mathds{R}),\ g\mapsto\rho(g)\\
    \rho^2:\ &\langle s\rangle
    \rightarrow\PSL(n,\mathds{R}),\ s^k\mapsto(\B\rho(s))^k.
\end{align*}
Together they induce a representation
\begin{equation*}
    \rho_{\B}:\pi_1(S)\rightarrow\PSL(n,\mathds{R})
\end{equation*}
since elements of $i_*(\pi_1(T_1,p_1))$ are actually powers of $\gamma$.

We call $\rho_{\B}$ the \emph{bending} of $\rho$ using $\B$.

\begin{lemma}
\label{lemma11}
The group $\rho_{\B}(\pi_1(S))$ lies in a conjugate of $\tau_n(\PSL(2,\mathds{R}))$ if and only if $\B\in\tau_n(\PGL(2,\mathds{R}))$.
\end{lemma}

\begin{proof}
Suppose that $\rho_{\B}(\pi_1(S))$ lies in $\PP\tau_n(\PSL(2,\mathds{R}))\PP^{-1}$ for some $\PP\in\PSL(n,\mathds{R})$. The group $\rho_{\B}(\pi_1(S_1))$ is Zariski-dense in $\tau_n(\PSL(2,\mathds{R}))$, since $j(\pi_1(S_1))$ is Zariski-dense in $\PSL(2,\mathds{R})$. Hence the equality $$\tau_n(\PSL(2,\mathds{R}))=\PP\tau_n(\PSL(2,\mathds{R}))\PP^{-1}.$$ Thus $\Int(\PP)$ induces an automorphism of $\tau_n(\PSL(2,\mathds{R}))$. We deduce that there exists $\X\in\PGL(2,\mathds{R})$ such that $\tau_n(\X)\PP^{-1}$ commutes with all elements of $\tau_n(\PSL(2,\mathds{R}))$. The latter is an absolutely irreducible subgroup of $\PSL(n,\mathds{R})$, hence Schur's lemma implies that $\PP=\tau_n(\X)$. Consequently $\rho_{\B}(\pi_1(S))$ lies in $\tau_n(\PSL(2,\mathds{R}))$.

Suppose now that $\gamma$ is separating. Then $\rho_{\B}(\pi_1(S_2))=\B\rho(\pi_1(S_2))\B^{-1}$ lies in $\B\tau_n(\PSL(2,\mathds{R}))\B^{-1}$ which is thus equal to $\tau_n(\PSL(2,\mathds{R}))$. By the above argument we get that $\B\in\tau_n(\PGL(2,\mathds{R}))$. If $\gamma$ is non-separating, then $\B\rho(s)\in\tau_n(\PSL(2,\mathds{R}))$ so $\B\in\tau_n(\PSL(2,\mathds{R}))$.
\end{proof}

\begin{lemma}
\label{lemma12}
The group $\rho_{\B}$ preserves a bilinear form represented by a matrix $\J\in\PGL(n,\mathbb{R})$ if and only if $\J_n=\B^\top\J_n\B$ in $\PGL(n,\mathds{R})$ and in that case $\J=\J_n$.
\end{lemma}

\begin{proof}
Suppose that $\rho_{\B}$ preserves $\J$. Then for all $s\in\pi_1(S_1)$ $$\J^{-1}\J_n=\rho(s)^{-1}\J^{-1}\J_n\rho(s)$$ so $\J^{-1}\J_n$ commutes with $\rho(\pi_1(S_1))$, and so with $\tau_n(\PSL(2,\mathds{R}))$. The latter is absolutely irreducible, so Schur's lemma implies that $\J$ is a scalar multiple of $\J_n$. Finally $\rho_{\B}$ preserves $\J_n$.

Assume that $\gamma$ is separating. For all $s\in\pi_1(S_2)$ we have
\begin{align*}
    &(\B\rho(s)\B^{-1})^\top\J_n(\B\rho(s)\B^{-1})=\J_n\\
    \implies&\rho(s)^\top\B^\top\J_n\B\rho(s)=\B^\top\J_n\B
\end{align*}
so $\rho(\pi_1(S_2))$ preserves $\J_n$ and $\B^\top\J_n\B$, which must be equal in $\PGL(n,\mathds{R})$ by the same argument as above. If $\gamma$ is non-separating, then
\begin{equation*}
    (\B\rho(s))^\top\J_n(\B\rho(s))=\J_n\ \implies\B^\top\J_n\B=\J_n.
\end{equation*}
\end{proof}

\begin{lemma}
\label{lemma13}
The group $\rho_{\B}(\pi_1(S))$ lies in a conjugate of \emph{$\textbf{G}_2(\mathds{R})$} if and only if $\B\in$\ \emph{$\textbf{G}_2(\mathds{R})$}.
\end{lemma}

\begin{proof}
Assume that $\rho_{\B}(\pi_1(S))\subset\PP\textbf{G}_2(\mathds{R})\PP^{-1}$ for some $\PP\in\PSL(n,\mathds{R})$.
Then $\rho(\pi_1(S_1))$ lies in $\PP\textbf{G}_2(\mathds{R})\PP^{-1}$ and so does $\tau_7(\PSL(2,\mathds{R}))$.
Thus $$\PP^{-1}\tau_7(\PSL(2,\mathds{R}))\PP$$ is a Lie subgroup of $\textbf{G}_2(\mathds{R})$. Let $\mathfrak{h}=\textrm{d}_{\I_7}\tau_7(\mathfrak{sl}_2(\mathds{R}))$ and $\mathfrak{g}_2$ be the Lie algebra of $\mathbf{G}_2(\mathds{R})$. The algebra $\mathfrak{h}$ is a principal $\mathfrak{sl}_2$-subalgebra of $\mathfrak{g}_2$ in the sense of Kostant \S5 \cite{Kostant_Principal3TDSBetti} since it acts irreducibly on $\mathds{R}^7$. There exists $g\in\textbf{G}_2(\mathds{R})$ such that $$\PP^{-1}\mathfrak{h}\PP=g\mathfrak{h}g^{-1}$$ since all automorphisms of $\mathfrak{g}_2$ are inner, as can be seen in Gündogan \cite{Gundogan_ComponentAutomorphismLieAlgebra}. We can deduce that $\X\mapsto\PP gX(\PP g)^{-1}$ induces an inner automorphism of $\mathfrak{h}$. There exists $\M\in\PSL(2,\mathds{R})$ such that for all $X\in\mathfrak{h}$ we have $g\PP X(g\PP)^{-1}=\tau_7(\M)X\tau_7(\M)^{-1}$ which implies that $\tau_7(\M)^{-1}g\PP$ commutes with all elements of $\mathfrak{h}$. We conclude that $g^{-1}\tau_7(\M)=\PP$ since $\mathfrak{h}$ acts absolutely irreducibly on $\mathds{R}^7$ and $\PP\in\textbf{G}_2(\mathds{R})$.
In particular, we have shown that $\rho_{\B}(\pi_1(S))\subset\textbf{G}_2(\mathds{R})$. Suppose that $\gamma$ is separating. Then $\B\rho(\pi_1(S_2))\B^{-1}\subset\textbf{G}_2(\mathds{R})$ and its Zariski-closure, which is $\B\tau_7(\PSL(2,\mathds{R}))\B^{-1}$, is also in $\textbf{G}_2(\mathds{R})$. By the same argument as above, we show that $\B\in\textbf{G}_2(\mathds{R})$. If $\gamma$ is non-separating, then $\B\rho(s)\in\textbf{G}_2(\mathds{R})$ so $\B\in\textbf{G}_2(\mathds{R})$.
\end{proof}

\section{Construction of Zariski-dense surface groups}

Let $\mathcal{O}^1$ be a torsion-free $\mathds{Q}$-arithmetic cocompact lattice of $\SL(2,\mathds{R})$. As shown in Section \ref{sectionQformsSL}, the irreducible embedding $\tau_n$ provides a representation of $\mathcal{O}^1$ that lies in the Hitchin component of the surface $\mathbb{H}^2/\mathcal{O}^1$ with image in a lattice $\Lambda$. We will bend this representation along a specific simple closed curve using a matrix in $\Lambda$ so that the image of the new representation will be Zariski-dense.

\subsection{The simple closed curve used to bend}

Let $A$ be a quaternion division algebra over $\mathds{Q}$ that splits over $\mathds{R}$. Up to isomorphism, it is of the form $(a,b)_{\mathds{Q}}$ with $a,b\in\mathds{N}$ not squares. Let $\{1,i,j,ij\}$ be a basis of $A$ such that $i^2=a$, $j^2=b$ and $ij=-ji$. We embed $A$ in $\M_2(\overline{\mathds{Q}})$ using the map defined by
\begin{equation}
\label{equationembedding}
    1\mapsto\begin{pmatrix}
    1&0\\
    0&1
    \end{pmatrix},\ i\mapsto\begin{pmatrix}
    \sqrt{a}&0\\
    0&-\sqrt{a}
    \end{pmatrix},\ j\mapsto\begin{pmatrix}
    0&\sqrt{b}\\
    \sqrt{b}&0
    \end{pmatrix},\ ij\mapsto\begin{pmatrix}
    0&\sqrt{ab}\\
    -\sqrt{ab}&0
    \end{pmatrix}.
\end{equation}
Via this embedding, $\SL(1,A)$ corresponds to matrices in $A$ of determinant $1$.

Let $\mathcal{O}$ be the order $\mathds{Z}[1,i,j,k]$. 
Then
\begin{equation*}
   \SL(1,\mathcal{O})\simeq\begin{Bmatrix}
   \begin{pmatrix}
   x_0+\sqrt{a}x_1&\sqrt{b}x_2+\sqrt{ab}x_3\\
   \sqrt{b}x_2-\sqrt{ab}x_3&x_0-\sqrt{a}x_1
   \end{pmatrix}\ |\ x_i\in\mathds{Z},\ \Det=1
   \end{Bmatrix}.
\end{equation*}
Diagonal elements of this group are in bijective correspondence with the solutions of the Pell equation $x_0^2-ax_1^2=1$ and are thus infinite in number by Dirichlet Unit's Theorem (see Theorem 0.4.2. in \cite{Maclachlan_ArithmeticHyperbolic3Manifolds}).

Let $S$ be a closed orientable surface of genus at least 2 for which there is a faithful representation $j:\pi_1(S)\to\PSL(1,\mathcal{O})=\SL(1,\mathcal{O})/\{\pm\I_2\}$.

\begin{lemma}
\label{lemma20}
There exists a simple closed curve on $S$ whose image under $j$ is diagonal for the canonical basis.
\end{lemma}

\begin{proof}
First of all, there exist diagonal elements in $j(\pi_1(S))$ because it is a finite index subgroup of $\PSL(1,\mathcal{O})$. Diagonal elements in $j(\pi(S))$ form a cyclic group. Let $\D$ be a generator and denote by $\gamma$ the closed curve it represents. Then $\gamma$ is simple if and only if its lifts to $\mathds{H}^2$ are disjoint from each other. We identify \begin{equation*}
\partial\mathds{H}^2=\mathds{P}(\mathds{R}^2)\xrightarrow{\sim}\hat{\mathds{R}},\ [x:y]\mapsto\left\{\begin{array}{ll}
    \frac{x}{y} & \mbox{if $y\neq0$}\\
    \infty & \mbox{otherwise.}
\end{array}\right.\end{equation*}
The curve $\gamma$ has a lift $\alpha$ which is the geodesic joining $0$ to $\infty$ since $\D$ is diagonal. To check if $\gamma$ is simple we only need to check that the images of this lift under $j(\pi(S))$ are either $\alpha$ or do not intersect $\alpha$. Let
\begin{equation*}
    g=\begin{pmatrix}
    x_0+\sqrt{a}x_1&\sqrt{b}x_2+\sqrt{ab}x_3\\
    \sqrt{b}x_2-\sqrt{ab}x_3&x_0-\sqrt{a}x_1
    \end{pmatrix}
\end{equation*}
be an element of $j(\pi_1(S))$ which is not diagonal. The image of $\alpha$ under $g$ is the geodesic joining $g(0)$ to $g(\infty)$. Note that neither of those can be $0$ or $\infty$ since $j(\pi_1(S))$ is a discrete subgroup. We need to prove that the product $g(0)g(\infty)$ is positive; computation gives
\begin{align*}
    g(0)g(\infty)=&\frac{(x_0+\sqrt{a}x_1)(\sqrt{b}x_2+\sqrt{ab}x_3)}{(\sqrt{b}x_2-\sqrt{ab}x_3)(x_0-\sqrt{a}x_1)}>0\\
    \Leftrightarrow\ &(x_0^2-ax_1^2)(bx_2^2-abx_3^2)>0\\
    \Leftrightarrow\ &(x_0^2-ax_1^2)(x_0^2-ax_1^2-1)>0
\end{align*}
for $\Det(g)=1$, and the latter is positive because the product of two consecutive integers is always positive or null.
\end{proof}

From now on, we denote by $\gamma$ the simple closed curve on $S$ whose image under $j$ is diagonal. For every $n\geq3$, $\tau_n(j(\gamma))$ is diagonal. The bending construction requires to find matrices that commute with $\tau_n(j(\gamma))$, i.e. which are diagonal, this is the goal of next part.

\subsection{Computation of the centralizer in $\Lambda$}

Since the algebraic group defined by $\HH(K)=\SL(1,A\otimes_{\mathds{Q}}K)$, for every field extension $K$ of $\mathds{Q}$, is a $\overline{\mathds{Q}}/\mathds{Q}$ form of $\SL_2$, one can compute the associated $1$-cocycle using the embedding described in (\ref{equationembedding}). Direct computations show that the $1$-cocycle associated to $\HH$ is $$\xi:\Gal(\overline{\mathds{Q}}/\mathds{Q})\to\PSL(2,\overline{\mathds{Q}}),\ \sigma\mapsto\Int(\T_{\sigma}^{a,b}).$$ All continuous $1$-cocycles $\eta:\Gal(\overline{\mathds{Q}}/\mathds{Q})\to\Aut(\SL_2(\overline{\mathds{Q}}))$ such that $\prescript{}{\eta}{\SL_2(\mathds{R})}\simeq\SL(2,\mathds{R})$ are equivalent to $\xi$ by Lemma \ref{lemma1}. Let $\zeta:\Gal(\overline{\mathds{Q}}/\mathds{Q})\to\Aut(\SL_n(\overline{\mathds{Q}}))$ be a $1$-cocycle $\tau_n$-compatible with $\xi$.

Let
\begin{equation}
\label{diag}
    \begin{pNiceMatrix}
    w_1&&\\
    &\Ddots&\\
    &&w_n
    \end{pNiceMatrix},\ \prod_{i} w_i=1
\end{equation} be a diagonal matrix.

\begin{lemma}
\label{diagonalelements}
If $\zeta$ is inner, the diagonal elements of $\prescript{}{\zeta}{\SL_n(\mathds{Q})}$ are of the form \eqref{diag} with $w_i\in\mathds{Q}(\sqrt{a})$ and $\sigma(w_i)=w_{n-i+1}$ for all $i$ with $\sigma\in\Gal(\mathds{Q}(\sqrt{a})/\mathds{Q})$ non-trivial.

If $\zeta$ is not inner let $\mathds{Q}(\sqrt{d})$ be the associated extension of $\mathds{Q}$.
\begin{enumerate}[label=\roman*), leftmargin=*]
    \item If $\sqrt{a}\in\mathds{Q}(\sqrt{d})$ then the diagonal elements of $\prescript{}{\zeta}{\SL_n(\mathds{Q})}$ are of the form \eqref{diag} with $w_i\in\mathds{Q}(\sqrt{a})$ and $w_i\sigma(w_i)=1$ for all $i$ with $\sigma\in\Gal(\mathds{Q}(\sqrt{a})/\mathds{Q})$ non-trivial.
    \item If $\sqrt{a}\notin\mathds{Q}(\sqrt{d})$ then the diagonal elements of $\prescript{}{\zeta}{\SL_n(\mathds{Q})}$ are of the form \eqref{diag} with $w_i\in\mathds{Q}(\sqrt{a},\sqrt{d})$, $\sigma(w_i)=w_{n-i+1}$ for all $i$ with $\sigma\in\Gal(\mathds{Q}(\sqrt{a},\sqrt{d})/\mathds{Q}(\sqrt{d}))$ non-trivial and $w_i\tau(w_{n-i+1})=1$ for all $i$ with $\tau\in\Gal(\mathds{Q}(\sqrt{a},\sqrt{d})/\mathds{Q}(\sqrt{a}))$ non-trivial.
\end{enumerate}

\end{lemma}

\begin{proof}
Since $\zeta$ is $\tau_n$-compatible with $\xi$, Proposition \ref{proposition1} describes $\zeta$ explicitly.

First, assume that $\zeta$ is inner. We are looking for diagonal matrices $\D\in\SL(n,\overline{\mathds{Q}})$ such that $\zeta(\theta)\circ\theta(\D)=\D$ for all $\theta\in\Gal(\overline{\mathds{Q}}/\mathds{Q})$, i.e. $$\tau_n(\T_{\theta})\theta(\D)\tau_n(\T_{\theta})^{-1}=\D.$$
If $\theta$ fixes $\sqrt{a}$, the previous equation implies that $\D$ has coefficients in $\mathds{Q}(\sqrt{a})$. If $\theta$ does not fix $\sqrt{a}$, then the equation shows that $\theta(w_{n-i+1})=w_{i}$ for all $i$. Hence the result.

Suppose now that $\zeta$ is not inner. We are looking for diagonal matrices $\D\in\SL(n,\overline{\mathds{Q}})$ such that $\zeta(\theta)\circ\theta(\D)=\D$ for all $\theta\in\Gal(\overline{\mathds{Q}}/\mathds{Q})$, i.e. $$\tau_n(\T_{\theta})\theta(\D)\tau_n(\T_{\theta})^{-1}=\D$$ whenever $\theta$ fixes $\sqrt{d}$ and $$\tau_n(\T_{\theta})\J_n^{-1}\theta(\D)^{-1}\J_n\tau_n(\T_{\theta})^{-1}=\D$$ whenever $\theta$ does not fix $\sqrt{d}$. 
If $\sqrt{a}\in\mathds{Q}(\sqrt{d})$, the first equation shows that $\D$ has coefficients in $\mathds{Q}(\sqrt{a})$ and the second that $w_i\theta(w_i)=1$ for $\theta\in\Gal(\mathds{Q}(\sqrt{a})/\mathds{Q})$ non-trivial.
Assume that $\sqrt{a}\notin\mathds{Q}(\sqrt{d})$. If $\theta$ fixes $\sqrt{d}$, then we conclude as in the previous case that $\D$ has coefficients in $\mathds{Q}(\sqrt{a},\sqrt{d})$. If $\theta$ does not fix $\sqrt{d}$ and does not fix $\sqrt{a}$, then the equation we get is $\theta(w_{n-i+1})^{-1}=w_{i}$ for all $i$. The case where $\theta$ fixes $\sqrt{a}$ but not $\sqrt{d}$ can be deduced from the previous equation. Hence the result.
\end{proof}

\subsection{Arithmetic properties of the trace}
Let $\Pi$ be a subgroup of $\GL(n,\mathds{R})$.
\begin{definition}
\label{definitiontrace}
Define $$\Tr(\Pi)=\{\Tr(\gamma)|\gamma\in\Pi\}.$$ If $\rho:\Pi\to\GL(n,\mathds{R})$ is a representations we denote $\Tr(\rho(\Pi))$ by $\Tr(\rho)$.
\end{definition}

In order to prove that two conjugacy classes of representations $\rho_1$ and $\rho_2$ of a surface group are not in the same orbit under the mapping class group it suffices to show that $\Tr(\rho_1)\neq\Tr(\rho_2)$.
This is how we will prove in the next part that there are infinitely many $\MCG$-orbits of thin Hitchin representations in a given lattice. To do so, we first study the arithmetic properties of $\Tr(\Pi)$, where $\Pi$ is a thin subgroup of a lattice, using the Strong-approximation Theorem which we now recall.

Let $\G$ be an algebraic subgroup of $\GL_n$ over $\mathds{Q}$. Denote by $I$ the ideal of $\mathds{Q}[X_{11},X_{12},...,X_{nn},\Det^{-1}]$ of polynomials vanishing on $\G$ and by $$I_0=I\cap\mathds{Z}[X_{11},X_{12},...,X_{nn},\Det^{-1}].$$ Let $p$ be a prime. We define an algebraic group over the finite field with $p$ elements $\mathds{F}_p$ by
$$\underline{\G}_{p}(K)=\Hom(\mathds{F}_p[X_{11},X_{12},...,X_{nn},\Det^{-1}]/I_0,K).$$

\begin{definition}
The algebraic group \emph{$\underline{\G}_{p}$} is called the \emph{reduction of $\G$ modulo $p$}.
\end{definition}

The latter depends on the embedding of $\G$ into $\GL_n$. In fact, the algebraic group $\underline{\G}_{p}$ is canonically defined for almost all primes $p$. Here by ``almost all" we mean ``all except finitely many". See Section 3.3 in \cite{Platonov_AlgebraicgroupsNumbertheory} for more details.

Let $\G$ be a connected (in the Zariski topology) semisimple $\mathds{Q}$-algebraic group. We say that $\G$ is \emph{simply connected} if for any connected $\mathds{Q}$-algebraic group $\HH$ and any surjective morphism $f:\G\to\HH$ of $\mathds{Q}$-algebraic groups with finite kernel, $f$ is an isomorphism. We say that $\G$ is \emph{absolutely almost simple} if $\G(\overline{\mathds{Q}})$ is not commutative and has no non-trivial connected normal subgroup.

\begin{theorem}[Strong-approximation, see \cite{Matthews_CongruencepropertiesZariskidense}]
\label{Strongapproximation}
Let $\G$ be a connected simply-connected absolutely almost simple algebraic group defined over $\mathds{Q}$. Let $\Pi$ be a Zariski-dense finitely generated subgroup of $\G(\mathds{Q})$. Then for almost all primes $p$ the reduction modulo $p$ of $\Pi$ equals $\underline{\G}_{p}(\mathds{F}_p)$.
\end{theorem}

Let $n\geq2$. Denote by $\SU(\I_n,\mathds{F}_p)$ the group $\{\M\in\SL(n,\mathds{F}_{p^2})|\overline{\sigma}(\M)^\top\M=\I_n\}$ with $\overline{\sigma}\in\Gal(\mathds{F}_{p^2}/\mathds{F}_p)$ non-trivial. Let $d\in\mathds{N}$ be not a square and $\sigma\in\Gal(\mathds{Q}(\sqrt{d})/\mathds{Q})$ non-trivial. For $p$ sufficiently large, the reduction of $\SU(\I_n,\sigma;\mathds{Q}(\sqrt{d}))$ modulo $p$ has for $\mathds{F}_p$-points
\begin{equation*}
    \underline{\SU(\I_n,\sigma;\mathds{Q}(\sqrt{d}))}_p(\mathds{F}_p)=\left\{\begin{array}{ll}
    \SL_n(\mathds{F}_p) & \mbox{if $d$ is a square modulo $p$} \\
    \SU(\I_n,\mathds{F}_p) & \mbox{if $d$ is not a square modulo $p$.}
\end{array}\right.
\end{equation*}
To see this, identify $\mathds{Q}(\sqrt{d})$ with $\mathds{Q}[\X]/(X^2-d)$ and note that elements of $\SU(\I_n,\sigma;\mathds{Q}(\sqrt{d}))$ are matrices $\M$ with coefficients in $\mathds{Q}[\X]/(\X^2-d)$ that satisfy
$$\partial(\M)^\top\M=\I_n$$ where $\partial$ is the $\mathds{Q}$-linear map that sends $\X$ to $-\X$. Note that
\begin{equation*}
    \mathds{F}_p[\X]/(\X^2-d)\simeq\left\{\begin{array}{ll}
    \mathds{F}_p\times\mathds{F}_p & \mbox{if $d$ is a square modulo $p$}\\
    \mathds{F}_{p^2} & \mbox{if $d$ is not a square modulo $p$.}
\end{array}\right.
\end{equation*} In the first case, $\partial$ induces the map 
\begin{align*}
    \mathds{F}_p\times\mathds{F}_p&\to\mathds{F}_p\times\mathds{F}_p\\
    (x,y)&\mapsto(y,x).
\end{align*}
Hence the reduction of $\SU(\I_n,\sigma;\mathds{Q}(\sqrt{d}))$ is
\begin{equation*}
    \{(\M,\N)\in\SL_n(\mathds{F}_p)\times\SL_n(\mathds{F}_p)\ |\ (\M,\N)(\N^\top,\M^\top)=(\I_n,\I_n)\}\simeq\SL_n(\mathds{F}_p)
\end{equation*}
by projection onto the first coordinate.
In the second case, $\partial$ induces the non-trivial field automorphism of $\mathds{F}_{p^2}$. Thus the group we get is $\SU(\I_n,\mathds{F}_p)$.

\begin{lemma}
\label{traceSL}
For any prime $p$ and $n\geq2$, $\Tr(\SL(n,\mathds{F}_p))=\mathds{F}_p$.
\end{lemma}

\begin{proof}
For $n=2$, it is proven by exhibiting a specific matrix. The result follows for any $n\geq2$ since we can embed $\SL(2,\mathds{F}_p)$ in $\SL(n,\mathds{F}_p)$ as a block.
\end{proof}

\begin{lemma}
\label{traceSU}
For any prime $p$ and $n\geq3$, $\Tr(\SU(\I_n,\mathds{F}_p))=\mathds{F}_{p^2}$.
\end{lemma}

\begin{proof}
We can embed $\SU(\I_3,\mathds{F}_p)$ in $\SU(\I_n,\mathds{F}_p)$ for any $n\geq3$ as a block. Hence it suffices to prove the result for $n=3$. Let $a\in\mathds{F}_{p^2}$. Since the map
$$\N:\mathds{F}_{p^2}\to\mathds{F}_p,\ x\mapsto x\overline{\sigma}(x)$$ is surjective, there exists $b\in\mathds{F}_{p^2}$ such that $\N(b)=-a-\overline{\sigma}(a)$. Let
$$\M=\begin{pmatrix}
a&b&1\\
\overline{\sigma}(b)&-1&0\\
1&0&0\\
\end{pmatrix}.$$ Then $\Det(\M)=1$ and $\overline{\sigma}(\M)^{\top}\J\M=\J$ where $$\J=\begin{pmatrix}
0&0&1\\
0&1&0\\
1&0&0
\end{pmatrix}.$$ Thus a conjugate of $\M$ lies in $\SU(\I_3,\mathds{F}_p)$. Since $\Tr(\M)=a-1$, we have $\Tr(\SU(\I_3,\mathds{F}_p))=\mathds{F}_{p^2}$.
\end{proof}

Let $n\geq3$ and $\A\in\SU(\I_n,\sigma;\mathds{Z}[\sqrt{d}])$. Then $\Tr(\A)\in\mathds{Z}[\sqrt{d}]$. The reduction of $\A$ modulo $p$ can be described in the following way. Identify $\mathds{Z}[\sqrt{d}]$ with $\mathds{Z}[\X]/(\X^2-d)$. Then $\A$ becomes a matrix whose entries are equivalence classes of polynomials. Since
\begin{equation*}
    \mathds{Z}[\X]/\!\raisebox{-.65ex}{\ensuremath{(\X^2-d)}/}\!\raisebox{-1.3ex}{\ensuremath{(p)}}\simeq\mathds{Z}[\X]/\!\raisebox{-.65ex}{\ensuremath{(p)}/}\!\raisebox{-1.3ex}{\ensuremath{(X^2-d)}}\simeq\mathds{F}_p[\X]/\!\raisebox{-.65ex}{\ensuremath{(\X^2-d)}},
\end{equation*}
the entries of the reduction of $\A$ modulo $p$ are elements of
\begin{equation*}
\left\{\begin{array}{ll}
    \mathds{F}_p\times\mathds{F}_p & \mbox{if $d$ is a square modulo $p$}\\
    \mathds{F}_{p^2} & \mbox{if $d$ is not a square modulo $p$.}
\end{array}\right.
\end{equation*}
If $d$ is a square modulo $p$ we consider the trace of the reduction of $\A$ to be an element of $\mathds{F}_p$ using the projection on the first coordinate.

\begin{proposition}
\label{propositionSUZ}
Let $n\geq3$ and $\Pi$ be a finitely generated subgroup of $\SU(\I_n,\sigma;\mathds{Z}[\sqrt{d}])$ which is Zariski-dense in $\SL(n,\mathds{R})$. Then the reduction of $\Tr(\Pi)\subset\mathds{Z}[\sqrt{d}]$ is equal to\begin{equation*}
\left\{\begin{array}{ll}
    \mathds{F}_p & \mbox{if $d$ is a square modulo $p$} \\
    \mathds{F}_{p^2} & \mbox{if $d$ is not a square modulo $p$}
\end{array}\right.
\end{equation*}for almost all primes $p$.
\end{proposition}

\begin{proof}
By Strong-approximation (Theorem \ref{Strongapproximation}), for almost all primes $p$, the reduction of $\Pi$ is equal to $\SL(n,\mathds{F}_p)$ or to $\SU(\I_n,\sigma;\mathds{Z}[\sqrt{d}])$ depending on $p$. Lemma \ref{traceSL} and \ref{traceSU} prove the result.
\end{proof}

Let $n$ be even. Let $A$ be a quaternion algebra over $\mathds{Q}$, $\overline{\phantom{s}}$ its conjugation and $\mathcal{O}$ an order of $A$. For $p$ sufficiently large, the reduction of the group $\SU(\I_{\frac{n}{2}},\overline{\phantom{s}}\otimes\sigma;A\otimes_{\mathds{Q}}\mathds{Q}(\sqrt{d}))$ is
\begin{equation*}
    \left\{\begin{array}{ll}
    \SL(n,\mathds{F}_p) & \mbox{if $d$ is a square modulo $p$} \\
    \SU(\I_n,\mathds{F}_p) & \mbox{if $d$ is not a square modulo $p$.}
\end{array}\right.
\end{equation*}

\begin{proposition}
\label{propositionSUO}
Let $n$ be even and $\Pi$ be a finitely generated subgroup of $\SU(\I_{\frac{n}{2}},\overline{\phantom{s}}\otimes\sigma;\mathcal{O}\otimes_{\mathds{Z}}\mathds{Z}[\sqrt{d}])$ which is Zariski-dense in $\SL(n,\mathds{R})$. Then the reduction of $\Tr(\Pi)\subset\mathds{Z}[\sqrt{d}]$ is equal to
\begin{equation*}
\left\{\begin{array}{ll}
    \mathds{F}_p & \mbox{if $d$ is a square modulo $p$} \\
    \mathds{F}_{p^2} & \mbox{if $d$ is not a square modulo $p$}
\end{array}\right.
\end{equation*}for almost all primes $p$.
\end{proposition}

\begin{proof}
The proof is the same as Proposition \ref{propositionSUZ}.
\end{proof}

Denote by $\Sp(n,\mathds{F}_p)$ the group
$\{\M\in\SL(n,\mathds{F}_{p})|\M^\top\J\M=\J\}$ where $\J$ is a block-diagonal matrix with $n$ block on the diagonal all equal to
$$\begin{pmatrix}
0&1\\
-1&0
\end{pmatrix}.$$ For $p$ sufficiently large, the reduction of $\SU(\I_{\frac{n}{2}},\overline{\phantom{s}};A)$ has for $\mathds{F}_p$-points $\Sp(n,\mathds{F}_p)$. To see this, recall that $$\SU(\I_{\frac{n}{2}},\overline{\phantom{s}};A)=\{\M\in\SL(\text{\tiny\(\frac{n}{2}\)},A)\ |\ \overline{\M}^\top\M=\I_{\frac{n}{2}}\}.$$ If we embed $A$ in $\M_2(\mathds{C})$ using $A\hookrightarrow A\otimes_{\mathds{Q}}\mathds{C}\simeq\M_2(\mathds{C})$ then the conjugation on $A$ agrees with $$\M_2(\mathds{C})\to\M_2(\mathds{C}),\ \X\mapsto\begin{pmatrix}
0&1\\
-1&0
\end{pmatrix}\X^\top\begin{pmatrix}
0&-1\\
1&0
\end{pmatrix}.$$ Hence the equation defining $\SU(\I_{\frac{n}{2}},\overline{\phantom{s}};A)$ becomes $\M^\top\J\M=\J$.

\begin{lemma}
\label{traceSp}
For any prime $p$ and $n\geq2$ even, $\Tr(\Sp(n,\mathds{F}_p))=\mathbb{F}_p$.
\end{lemma}

\begin{proof}
It follows from Lemma \ref{traceSL} and from the embedding of $\SL(2,\mathds{F}_p)$ in $\Sp(n,\mathds{F}_p)$ as a block.
\end{proof}

\begin{proposition}
\label{propositionSp}
Let $n\geq2$ be even and $\Pi$ be a finitely generated subgroup of $\SU(\I_{\frac{n}{2}},\overline{\phantom{s}};\mathcal{O})$ which is Zariski-dense in $\Sp(n,\mathds{R})$. Then the reduction of $\Tr(\Pi)\subset\mathds{Z}$ modulo $p$ is equal to $\mathds{F}_p$ for almost all primes $p$.
\end{proposition}

\begin{proof}
By the Strong-approximation Theorem, for almost all primes the reduction of $\Pi$ is equal to $\Sp(n,\mathds{F}_p)$. Lemma \ref{traceSp} proves that the reduction of $\Tr(\Pi)$ is all of $\mathds{F}_p$ for almost all primes
\end{proof}

Let $n\geq4$ and $p\neq2$ be prime. Denote by $\SO(\I_n,\mathds{F}_p)$ the group $\{\M\in\SL(n,\mathds{F}_p)|\M^\top\M=\I_n\}$ and by $\Omega(\I_n,\mathds{F}_p)$ the commutator subgroup of $\SO(\I_n,\mathds{F}_p)$. Let $\Q\in\SL(n,\mathds{Q})$ be a symmetric matrix. If $n$ is odd, for $p$ sufficiently large, the reduction of $\SO(\Q,\mathds{Q})$ modulo $p$ has for $\mathds{F}_p$-points $\SO(\I_n,\mathds{F}_p)$.

\begin{lemma}
\label{traceOmega}
For any prime $p\neq2$ and $n\geq4$, $\Tr(\Omega(\I_n,\mathds{F}_p))=\mathds{F}_p$.
\end{lemma}

\begin{proof}
We can embed $\Omega(\I_4,\mathds{F}_p)$ in $\Omega(\I_n,\mathds{F}_p)$ as a block. Hence is suffices to prove the result for $n=4$. Pick $a\in\mathds{F}_p^{\times}$. Let
$$\M=\begin{pmatrix}
0&0&-a&a\\
0&0&1&0\\
1&1&0&0\\
a^{-1}&0&0&0
\end{pmatrix}.$$ We can check that $\Det(\M)=1$ and that $\M^\top\J\M=\J$ where
$$\J=\begin{pmatrix}
0&0&0&1\\
0&0&1&0\\
0&1&0&0\\
1&0&0&0
\end{pmatrix}.$$ Since $\Det(\J)=1$, $\J$ is congruent to $\I_4$. It follows that a conjugate of $\M$ lies in $\SO(\I_4,\mathds{F}_p)$. As we can see in \cite{Suzuki_GroupTheoryI} paragraph 5 chapter 3, $\Omega(\I_4,\mathds{F}_p)$ is of index 2 in $\SO(\I_4,\mathds{F}_p)$. Hence a conjugate of $\M^2$ lies in $\Omega(\I_4,\mathds{F}_p)$. Since $\Tr(\M^2)=-2a+4$ and $\Tr(\I_4)=4$, $\Tr(\Omega(\I_4,\mathds{F}_p))=\mathds{F}_p$.
\end{proof}

\begin{proposition}
\label{propositionSO}
Let $n\geq5$ be odd and $\Q\in\SL(n,\mathds{Q})$ a symmetric matrix. Let $\Pi$ be a finitely generated subgroup of $\SO(\Q,\mathds{Z})$ which is Zariski-dense in $\SO(\Q,\mathds{R})$. Then the reduction of $\Tr(\Pi)\subset\mathds{Z}$ modulo $p$ is equal to $\mathds{F}_p$ for almost all primes $p$.
\end{proposition}

\begin{proof}
Since the algebraic group defined by $\G(K)=\SO(\Q,K)$ is not simply-connected, we cannot apply Theorem \ref{Strongapproximation}. Nevertheless Theorem 5.1 in Nori \cite{Nori_SubgroupsofGLn(Fp)} shows that for almost all primes $p$, the reduction of $\Pi$ modulo $p$ contains $\Omega(\I_n,\mathds{F}_p)$\footnote{Indeed Remark 3.6 in \cite{Nori_SubgroupsofGLn(Fp)} shows that $A(\mathds{F}_p)^{+}$ is of index 2 in $\SO(\I_n,\mathds{F}_p)$. Note also that $\Omega(\I_n,\mathds{F}_p)$ is of index 2 in $\SO(\I_n,\mathds{F}_p)$, as we can see in \cite{Suzuki_GroupTheoryI} paragraph 5 chapter 3. Being the commutator subgroup, $\Omega(\I_n,\mathds{F}_p)$ is the only index 2 subgroup of $\SO(\I_n,\mathds{F}_p)$.}. Lemma \ref{traceOmega} proves that the reduction of $\Tr(\Pi)$ is equal to $\mathds{F}_p$ for almost all primes $p$.
\end{proof}

\subsection{Certifying the Zariski-density of the bending}

We can now prove Theorems \ref{theorem1} and \ref{theorem2}. Let $A$ be a quaternion division algebra over $\mathds{Q}$ with basis $\{1,i,j,k\}$ satisfying $i^2=a$ and $j^2=b$, that splits over $\mathds{R}$ and let $\mathcal{O}=\mathds{Z}[1,i,j,k]$. We embed $\mathcal{O}$ in $\M_2(\overline{\mathds{Q}})$ using the map defined in (\ref{equationembedding}). Then $\mathcal{O}^1$ is commensurable with $\prescript{}{\xi}{\SL_2(\mathds{Z})}$ for $\xi:\Gal(\overline{\mathds{Q}}/\mathds{Q})\to\PSL(2,\overline{\mathds{Q}})$ a $1$-cocycle. Let $S$ be a closed orientable surface of genus at least 2 and suppose there exists a faithful representation $j:\pi_1(S)\to\PSL(1,\mathcal{O})$. Denote by $\gamma$ the simple closed curve on $S$ which image under $j$ is diagonal, as provided by Lemma \ref{lemma20}. Let $\rho=\tau_n\circ j$. Let $\zeta:\Gal(\overline{\mathds{Q}}/\mathds{Q})\to\Aut(\SL_n(\overline{\mathds{Q}}))$ be a $1$-cocycle $\tau_n$-compatible with $\xi$ (see Definition \ref{definitioncompatible}).

We bend $\rho$ along the curve $\gamma$ (see Section \ref{sectionbending} for the description of bending). When it's possible, we exhibit a diagonal matrix $\B$ with positive coefficients which is in $\prescript{}{\zeta}{\SL_n(\mathds{Z})}$, so that $\rho_{\B}(\pi_1(S))\subset \PP(\prescript{}{\zeta}{\SL_n(\mathds{Z})})$ but such that $\rho_{\B}(\pi_1(S))$ is Zariski-dense in $\PSL(n,\mathds{R})$. Since $\B$ has positive coefficients, there is a continuous path from $\I_n$ to $\B$ inside the centralizer of $\gamma$ and thus there is a continuous path from $\rho$ to $\rho_{\B}$ which guarantees that $\rho_{\B}$ is a Hitchin representation. We can thus use the classification of the Zariski-closures of Hitchin representations.

\begin{theorem}[Sambarino \cite{Sambarino_InfinitesimalZariskiClosure}]
\label{theoremGuichard}
Let $\rho:\pi_1(S)\to\PSL(n,\mathds{R})$ be a Hitchin representation. Then the Zariski-closure of $\rho$ is either $\PSL(n,\mathds{R})$ or conjugate to one of the following:
\begin{itemize}
    \setlength\itemsep{0cm}
    \item $\tau_n(\PGL(2,\mathds{R}))\cap\PSL(n,\mathds{R})$
    \item $\PSp(2k,\mathds{R})$ if $n=2k$ for all $k\geq1$,
    \item $\SO(\I_{k+1,k},\mathds{R})$ if $n=2k+1$ for all $k\geq1$,
    \item \emph{$\textbf{G}_2(\mathds{R})$} if $n=7$.
\end{itemize}
\end{theorem}

We could also use a weaker version of this theorem. Since our representations are constructed via bending, it is easy to see that their Zariski-closures contain a principal $\SL_2(\mathds{R})$. The classification of their Zariski-closures then follows from a paper by Dynkin (see tables 13 and 14 in \cite{Dynkin_SemisimplesubagelbrasofsemisimpleLiealgebras}, \cite{Sambarino_InfinitesimalZariskiClosure} also gives a proof of this).

Since $a$ is not a square, $\mathds{Z}[\sqrt{a}]^{\times}\simeq\langle\pm1\rangle\times\langle \omega\rangle$ by Dirichlet's Unit Theorem (Theorem 0.4.2 in \cite{Maclachlan_ArithmeticHyperbolic3Manifolds}), where $\omega$ is the fundamental unit of $\mathds{Z}[\sqrt{a}]$. Note that $\omega\sigma(\omega)=\pm1$.

\begin{proof}[Proof of Theorem \ref{theorem1}]
Let $n\geq3$ be odd and assume that $\zeta$ is not inner. Let $\mathds{Q}(\sqrt{d})$ be the quadratic extension of $\mathds{Q}$ associated to $\zeta$ with $d\in\mathds{N}$. Proposition \ref{proposition23} implies that $\prescript{}{\zeta}{\SL_n(\mathds{Z})}$ is commensurable with a conjugate of the group $\SU(\I_n,\sigma,\mathds{Z}[\sqrt{d}])$.

To construct thin Hitchin representations in $\prescript{}{\zeta}{\SL_n(\mathds{Z})}$  we bend $\rho$ using a diagonal matrix with positive coefficients $\B\in\prescript{}{\zeta}{\SL_n(\mathds{Z})}$ which is not in $\SO(\J_n,\mathds{R})$. Given such a matrix, Theorem \ref{theoremGuichard} and Lemma \ref{lemma12} then implies that $\rho_{\B}(\pi_1(S))$ is Zariski-dense in $\SL(n,\mathds{R})$. 
If $\sqrt{a}\in\mathds{Q}(\sqrt{d})$, one can bend by the matrix $$\B=\Diag(\omega^4,1,...,1,\omega^{-2},\omega^{-2})\in\prescript{}{\zeta}{\SL_n}(\mathds{Z}),$$ which lies in $\prescript{}{\zeta}{\SL_n}(\mathds{Z})$ by Lemma \ref{diagonalelements}. Suppose that $\sqrt{a}\not\in\mathds{Q}(\sqrt{d})$. Let $\sigma,\tau\in\Gal(\mathds{Q}(\sqrt{a},\sqrt{d})/\mathds{Q})$ be the automorphisms defined by 
\begin{align*}
    \sigma(\sqrt{a})=-\sqrt{a},\ &\sigma(\sqrt{d})=\sqrt{d}\\
    \tau(\sqrt{a})=\sqrt{a},\ &\tau(\sqrt{d})=-\sqrt{d}.
\end{align*}
Let $\theta$ be a fundamental unit in $\mathds{Z}[\sqrt{d}]^{\times}$. Since $\theta\sigma(\theta)\tau(\theta)\sigma(\tau(\theta))=\pm1$ and $\sigma(\theta)=\theta$, we deduce that $\theta^2\tau(\theta)^2=1$ and $\theta^2\sigma(\theta^2)=\theta^4\neq1$. It follows that $$\B=\Diag(\theta^2,1,...,1,\theta^{-4},1,...,1,\theta^2)\in\prescript{}{\zeta}{\SL_n(\mathds{Z})},$$ see Lemma \ref{diagonalelements}. We can use the latter matrix to bend.

In both cases, the sequence of bendings $(\rho_{\B^k})_{k}$ gives infinitely many $\MCG$-orbits of thin Hitchin representations. Indeed, suppose that there are only finitely many such orbits. By Theorem \ref{Strongapproximation} and Proposition \ref{propositionSUZ}, there exists a finite set of primes $\mathcal{V}$ such that for every prime $p$ not in $\mathcal{V}$ and any $k\geq 1$, the reduction of $\Tr(\rho_{\B^k})$ modulo $p$ is equal to
\begin{equation*}
    \left\{\begin{array}{ll}
        \mathds{F}_p & \mbox{if $d$ is a square modulo $p$} \\
        \mathds{F}_{p^2} & \mbox{if $d$ is not a square modulo $p$.}
    \end{array}\right.
\end{equation*}
There exists a polynomial $\PP\in\mathds{Z}[\X]$ such that $\Tr(\tau_n(\A))=\PP(\Tr(\A))$. For any prime $p$, the polynomial $\PP$ induces a map
$$\mathds{F}_p\to\mathds{F}_p,\ x\mapsto\PP(x)$$ and this map is not surjective for infinitely many primes (see Corollary 1.8 in \cite{Shallue_PermutationPolynomialsFiniteFields} for example). Pick a prime $p\not\in\mathcal{V}$ for which the map induced by $\PP$ is not surjective. Choose $k\geq 1$ such that $\B^k$ is not in $\SO(\J_n,\mathds{R})$ but is trivial modulo $p$, which exists since every element of a finite group has finite order. Then modulo $p$, $\Tr(\rho_{\B^k})\equiv\Tr(\rho)\equiv\PP(\Tr(j))[p]$ which is not $\mathds{F}_p$. This is a contradiction.

\end{proof}

\begin{proof}[Proof of corollary \ref{corollary1}]
It is a consequence of Theorem \ref{theorem1}, Theorem 1.1 in \cite{Long_ZariskidensesurfaceSL2k+1Z} and Proposition \ref{propositionprime}.
\end{proof}

\begin{proof}[Proof of Theorem \ref{theorem2}]
Let $n\geq4$ be even and suppose that $\zeta$ is not inner. Let $\mathds{Q}(\sqrt{d})$ be the quadratic extension of $\mathds{Q}$ associated to $\zeta$ with $d\in\mathds{N}$. Proposition \ref{proposition23} shows that $\prescript{}{\zeta}{\SL_n(\mathds{Z})}$ is commensurable with a conjugate of $\SU(\I_n,\sigma;\mathds{Z}[\sqrt{d}])$ (resp. $\SU(\I_{\frac{n}{2}},\overline{\phantom{s}}\otimes\sigma;\mathcal{O}\otimes\mathds{Z}[\sqrt{d}])$) if $\sqrt{a}\in\mathds{Q}(\sqrt{d})$ (resp. $\sqrt{a}\notin\mathds{Q}(\sqrt{d})$).

To construct thin Hitchin representations in $\prescript{}{\zeta}{\SL_n(\mathds{Z})}$  we bend $\rho$ using a diagonal matrix with positive entries $\B\in\prescript{}{\zeta}{\SL_n(\mathds{Z})}$ which is not in $\Sp(\J_n,\mathds{R})$. 
If $\sqrt{a}\in\mathds{Q}(\sqrt{d})$, we can use the matrix $$\B=\Diag(\omega^4,1,...,1,\omega^{-2},\omega^{-2}).$$
Otherwise, we can use the matrix $$\B=\Diag(\theta^2,\theta^{-2},1,...,1,\theta^{-2},\theta^2).$$ where $\theta$ is a fundamental unit in $\mathds{Z}[\sqrt{d}]^{\times}$. 
In both cases, the sequence of bendings $(\rho_{\B^k})_{k}$ gives infinitely many $\MCG$-orbits of thin Hitchin representations by the same argument as in the proof of Theorem \ref{theorem1}.



\end{proof}



We can now prove Theorems \ref{theorem3}, \ref{theorem4} and \ref{theorem5}. We use the same strategy as above. For some $\mathds{R}/\mathds{Q}$-forms $\prescript{}{\zeta}{\G}$ of a given algebraic split real Lie group $\G$ we exhibit a matrix $$\B\in\prescript{}{\zeta}{\G}(\mathds{Z})=\prescript{}{\zeta}{\SL_n(\mathds{Z})}\cap\G(\mathds{R})$$ which is diagonal (so that it commutes with $\rho(\gamma)$) with positive entries and such that $\rho_{\B}(\pi_1(S))$ is Zariski-dense in the corresponding Lie group. Note that, even thought there are several $1$-cocycles with value in $\Aut(\SL_n(\overline{\mathds{Q}}))$ which are $\tau_n$-compatible with $\xi$, they all induce the same $1$-cocycle on $\SO(\J_n,\overline{\mathds{Q}})$, $\textbf{G}_2(\overline{\mathds{Q}})$ and $\Sp(\J_n,\overline{\mathds{Q}})$.

\begin{proof}[Proof of Theorem \ref{theorem3}]
Assume that $n\geq5$ is odd. To construct thin Hitchin representations in $\prescript{}{\zeta}{\SO(\J_n)(\mathds{Z})}$ we bend $\rho$ using a diagonal matrix with positive entries $\B\in\prescript{}{\zeta}{\SO(\J_n)(\mathds{Z})}$ which is not in $\tau_n(\PGL(2,\mathds{R}))$ (nor in $\textbf{G}_2(\mathds{R})$ if $n=7$). Given such a matrix, Theorem \ref{theoremGuichard} and Lemma \ref{lemma11} show that $\rho_{\B}(\pi_1(S))$ is Zariski-dense in $\SO(\J_n,\mathds{R})$. For instance, one can use the matrix $$\B=\Diag(\omega^2,\sigma(\omega)^2,1,\ldots,1,\omega^2,\sigma(\omega)^2).$$

Using Theorem \ref{propositionSO}, one can show that the sequence $(\rho_{B^k})_k$ gives rise to infinitely many $\MCG$-orbits of thin Hitchin representations, in the same manner as in the proof of Theorem \ref{theorem1}.


If $n\equiv\pm1[8]$, then $\prescript{}{\zeta}{\SO(\J_n)}(\mathds{Z})$ is commensurable with a conjugate of $\SO(\J_n,\mathds{Z})$ by Lemma \ref{Qformsorthogonal}. 
Otherwise, Lemma \ref{Qformsorthogonal} shows that $\prescript{}{\zeta}{\SO(\J_n)}(\mathds{Z})$ is commensurable with a conjugate of $\SO(\Q,\mathds{Z})$ where $\Q\in\SL(n,\mathds{Q})$ is a symmetric matrix of signature equal to the signature of $\J_n$ and of Hasse invariant $\mathcal{E}_p(\Q)=(a,b)_{\mathds{Q}_p}\otimes(-1,-1)_{\mathds{Q}_p}$ for all primes $p$.

Suppose that $n\equiv\pm3[8]$. To conclude the proof, we claim that the above construction yields thin Hitchin representations in all $\SO(\Q,\mathds{Z})$ for $\Q\in\SL(2k+1,\mathds{Q})$ a symmetric matrix of signature $(k+1,k)$ not equivalent to $\I_{k+1,k}$. Let $\Q$ be such a symmetric matrix. Denote by $\Omega$ the set of prime numbers $p$ such that $$\mathcal{E}_p(\Q)\otimes(-1,-1)_{\mathds{Q}_p}\neq 1.$$ This set has even cardinality by Hilbert's Reciprocity Law, see Theorem 0.9.10 in \cite{Maclachlan_ArithmeticHyperbolic3Manifolds}. Theorem 7.3.6 in \cite{Maclachlan_ArithmeticHyperbolic3Manifolds} now implies that there exists a quaternion algebra $A$ over $\mathds{Q}$ that splits over $\mathds{R}$ and splits over $\mathds{Q}_p$ if and only if $p\not\in\Omega$. Since $\Q$ is not equivalent to $\I_{k+1,k}$, $\Omega$ is non-empty. In particular, $A\not\simeq\mathrm{M}_2(\mathds{Q})$.
Using this quaternion algebra in the construction described above yields thin Hitchin representations in $\SO(\Q,\mathds{Z})$.
\end{proof}

\begin{proof}[Proof of Theorem \ref{theorem5}]
Let $n\geq4$ be even. Then $\prescript{}{\zeta}{\Sp(\J_n)(\mathds{Z})}$ is commensurable with a conjugate of $\SU(\I_{\frac{n}{2}},\overline{\phantom{s}};\mathcal{O})$. To construct thin Hitchin representations in $\prescript{}{\zeta}{\Sp(\J_n)(\mathds{Z})}$ we bend $\rho$ with a diagonal matrix with positive entries $\B\in\prescript{}{\zeta}{\Sp(\J_n)(\mathds{Z})}$ which is not in $\tau_n(\PGL(2,\mathds{R}))$. Given such a matrix, Theorem \ref{theoremGuichard} and Lemma \ref{lemma11} show that $\rho_{\B}(\pi_1(S))$ is Zariski-dense in $\Sp(\J_n,\mathds{R})$. For instance, one can use the matrix $$\B=\Diag(\omega^2,...,\omega^2,\sigma(\omega)^2,...,\sigma(\omega)^2).$$
Using Theorem \ref{propositionSp}, one can show that the sequence $(\rho_{B^k})_k$ gives rise to infinitely many $\MCG$-orbits of thin Hitchin representations, in the same manner as in the proof of Theorem \ref{theorem1}.

\end{proof}

\begin{proof}[Proof of Corollary \ref{corollary2}]
It is a consequence of Theorem \ref{theorem5} and Proposition \ref{propositionsymplectic}.
\end{proof}

\begin{proof}[Proof of Theorem \ref{theorem4}]
Suppose $n=7$. To construct thin Hitchin representations in $\textbf{G}_2(\mathds{Z})$ we bend $\rho$ using a diagonal matrix with positive entries $\B\in\textbf{G}_2(\mathds{Z})$ which is not in $\tau_7(\PGL(2,\mathds{R}))$. Given such a matrix, Theorem \ref{theoremGuichard} and Lemma \ref{lemma11} show that $\rho_{\B}(\pi_1(S))$ is Zariski-dense in $\textbf{G}_2(\mathds{R})$.
For instance, one can use the matrix $$\B=\Diag(\omega^2,\omega^2,1,1,1,\sigma(\omega)^2,\sigma(\omega)^2).$$
Using Lemma 5.12 in \cite{Audibert_ZariskidenseHitchininuniform}, one can show that the sequence $(\rho_{B^k})_k$ gives rise to infinitely many $\MCG$-orbits of thin Hitchin representations, in the same manner as in the proof of Theorem \ref{theorem1}.
\end{proof}

\bibliography{main}

\end{document}